\documentclass{aims}
\usepackage{amsmath}
\usepackage{paralist}
\usepackage{graphics} 
\usepackage{epsfig}
\usepackage{graphicx}  
\usepackage{physics}
\usepackage{mathtools}
\usepackage{epstopdf} 
\usepackage[colorlinks=true]{hyperref}
\usepackage{dsfont}
\hypersetup{urlcolor=blue, citecolor=red}

\numberwithin{equation}{section}

\def\currentvolume{X}

\def\currentissue{X}
\def\currentyear{200X}
\def\currentmonth{XX}
\def\ppages{X--XX}

\newtheorem{theorem}{Theorem}[section]
\newtheorem{corollary}{Corollary}[section]
\newtheorem{assumption}{Assumption} 

\newtheorem{lemma}{Lemma}[section]
\newtheorem{proposition}{Proposition}[section]

\newtheorem{definition}{Definition}[section]
\newtheorem{remark}{Remark}[section]

\title[STATE-CONSTRAINED CONTROL FOR JUMP DIFFUSIONS] 
      {Backward Reachability Approach to State-Constrained Stochastic Optimal Control Problems for Jump Diffusion Systems}

\author[Jun Moon]{}

\subjclass{Primary: 45K05, 49L25; Secondary: 93E20.}
 \keywords{State-constrained problems, stochastic target problems, jump diffusion systems, integro-partial differential equations, viscosity solutions.}

 \email{jmoon12@uos.ac.kr}

\thanks{This research was supported in part by the National Research Foundation of Korea (NRF) Grant funded by the Ministry of Science and ICT, South Korea (NRF-2017R1E1A1A03070936, NRF-2017R1A5A1015311).
 }

\begin{document}
\maketitle

\centerline{\scshape Jun Moon}
\medskip
{\footnotesize
 \centerline{School of Electrical and Computer Engineering}
   \centerline{University of Seoul, Seoul, 02504, South Korea}
} 

%

\bigskip

 \centerline{(Communicated by the associate editor name)}

\begin{abstract}
In this paper, we consider the stochastic optimal control problem for jump diffusion systems with state constraints. In general, the value function of such problems is a discontinuous viscosity solution of the Hamilton-Jacobi-Bellman (HJB) equation, since the regularity cannot be guaranteed at the boundary of the state constraint. By adapting approaches of \cite{Bokanowski_SICON_2016} and the stochastic target theory, we obtain an equivalent representation of the original value function as the backward reachable set. We then show that this backward reachable can be characterized by the zero-level set of the auxiliary value function for the unconstrained stochastic control problem, which includes two additional unbounded controls as a consequence of the martingale representation theorem. We prove that the auxiliary value function is a unique continuous viscosity solution of the associated HJB equation, which is the second-order nonlinear integro-partial differential equation (IPDE). Our paper provides an explicit way to characterize the original (possibly discontinuous) value function as a zero-level set of the continuous solution of the auxiliary HJB equation. The proof of the existence and uniqueness requires a new technique due to the unbounded control sets, and the presence of the singularity of the corresponding L\'evy measure in the nonlocal operator of the HJB equation.
\end{abstract}

\section{Introduction}\label{Section_1}

Let $B$ and $\tilde{N}$ be a standard Brownian motion and an $E$-marked compensated Poisson random process, respectively, which are mutually independent of each other. The problem studied in this paper is to minimize the following objective functional over $u \in \mathcal{U}_{t,T}$
	\begin{align}
\label{eq_0_1}	
J(t,a;u) = \mathbb{E} \Bigl [ \int_t^T l(s,x_{s}^{t,a;u},u_s) \dd s + m(x_{T}^{t,a;u})	\Bigr ],
\end{align}
subject to the $\mathbb{R}^n$-dimensional stochastic differential equation (SDE)
\begin{align}
\label{eq_0_2}
\begin{cases}
\dd x_{s}^{t,a;u} = f(s,x_{s}^{t,a;u},u_s)\dd s + \sigma(s,x_{s}^{t,a;u},u_s) \dd B_s	\\
\qquad \qquad + \int_{E} \chi(s,x_{s-}^{t,a;u},u_s,e) \tilde{N}(\dd e, \dd s), ~ s \in (t,T]\\
x_{t}^{t,a;u} = a,
\end{cases}	
\end{align}
and the \emph{state constraint} ($\Gamma$ is a nonempty closed subset of $\mathbb{R}^n$)
\begin{align}
\label{eq_0_3}
x_{s}^{t,a;u} \in \Gamma,~ \forall s \in [t,T],~ \text{$\mathbb{P}$-a.s.}
\end{align}
The precise problem formulation is given in Section \ref{Section_2_1}. The associated value function for (\ref{eq_0_1}) is defined by
\begin{align}
\label{eq_0_4}
V(t,a) := \inf_{u \in \mathcal{U}_{t,T}} \{ J(t,a;u)~|~ x_{s}^{t,a;u} \in \Gamma,~ \text{$\mathbb{P}$-a.s.,}~ \forall s \in [t,T] \}.	
\end{align}
The problem in (\ref{eq_0_4}) can then be referred to as the \emph{stochastic optimal control problem for jump diffusion systems with state constraints}. 

The first main result of this paper is that (\ref{eq_0_4}) can be equivalently represented by (see Theorem \ref{Theorem_1})
	\begin{align}
	\label{eq_0_5}
		V(t,a) = \inf \{b \geq 0~| (a,b) \in \mathcal{R}_t^\Gamma \} = \inf \{b \geq 0~|~ W(t,a,b) = 0, a \in \mathbb{R}^n \},
	\end{align}
	where $\mathcal{R}_t^\Gamma$ is the \emph{backward reachable set} of the stochastic target  problem with state constraints (see (\ref{eq_0_6_1_1})), and $W$ is a continuous value function of the auxiliary stochastic control problem that includes unbounded control sets $\mathcal{A}_{t,T} \times \mathcal{B}_{t,T}$. Then our second main result is that $W$ is a unique continuous viscosity solution of the following Hamilton-Jacobi-Bellman (HJB) equation with suitable boundary conditions (see Theorems \ref{Theorem_2} and \ref{Theorem_3}): (time and state arguments are suppressed)
\begin{align}	
\label{eq_0_6}
& - \partial_t W	+ \mathop{\sup_{u \in U}}_{\alpha \in \mathbb{R}^{r}, \beta \in G^2} \Bigl \{ - \langle D W, \begin{bmatrix}
 f(u) \\
 -l(u)	
 \end{bmatrix} \rangle  - \frac{1}{2} \Tr 
\Bigl ( \begin{bmatrix}
\sigma \sigma^\top(u) & \sigma(u) \alpha \\
(\sigma(u) \alpha)^\top & \alpha^\top \alpha 	
\end{bmatrix} D^2 W \Bigr )
\\
& - \int_{E} \bigl [ W(t,a + \chi(u,e), b+ \beta(e)) - W(t,a,b) - \langle DW,  \begin{bmatrix}
  \chi(u,e) \\
  \beta(e)	
  \end{bmatrix} \rangle \bigr ]\pi(\dd e) \Bigr \} - d(a,\Gamma), \nonumber
\end{align}
which is the second-order nonlinear integro-partial differential equation (IPDE) that includes two unbounded control variables $(\alpha,\beta) \in \mathbb{R}^r \times G^2$. We give a detailed statement on the main results of the paper after providing the literature review.

(Deterministic and stochastic) control problems with state constraints were studied extensively in the literature; see \cite{Soner_SICON_1986_1, Soner_SICON_1986_2, Capuzzo_AMS_1990, Hermosilla_SCL_2017, Frankowska_CV_2012, Bardi_DCDS_2000, Altarovici_ESIAM_2013, Katsoulakis_Indiana_1994, Ishii_Indiana_2002, Bouchard_Elie_SICON_2010, Bouchard_Nutz_SICON_2012} and the references therein. 
In particular, as discussed in \cite{Soner_SICON_1986_1, Katsoulakis_Indiana_1994, Bouchard_Nutz_SICON_2012, Bokanowski_SICON_2016}, it can be conjectured that $V$ in (\ref{eq_0_5}) is only a discontinuous viscosity solution of the following constrained HJB equation:
\begin{align}
\label{eq_0_6_2_1_2_1_1}
\begin{cases}
\partial_t V + \sup_{u \in U} \bigl \{ - \langle D V, f(u) \rangle - \frac{1}{2} \Tr ( \sigma \sigma^\top(u) D^2 V) - l(u) \\
\quad \quad - \int_{E} [ V(t,a+\chi(u,e)) - V(t,a) - \langle D V,\chi(u,e) \rangle ] \pi (\dd e) \bigr \} = 0 \\
\qquad \qquad t \in [0,T),~ x \in \mathrm{int}(\Gamma)\\
\partial_t V + \sup_{u \in U} \bigl \{ - \langle D V, f(u) \rangle - \frac{1}{2} \Tr ( \sigma \sigma^\top(u) D^2 V) - l(u) \\
\quad \quad - \int_{E} [ V(t,a+\chi(u,e)) - V(t,a) - \langle D V,\chi(u,e) \rangle ] \pi (\dd e) \bigr \} \geq 0 \\
\qquad \qquad t \in [0,T),~ x \in \partial \Gamma \\
V(T,x) = m(x),~ x \in \Gamma.
\end{cases}
\end{align}
However, it may be hard to characterize the solution of (\ref{eq_0_6_2_1_2_1_1}) directly due to discontinuity and inequality constraint at $\partial \Gamma$ (the boundary of $\Gamma$). We also note that the references mentioned above considered state-constrained problems only for deterministic systems or SDEs in a Brownian setting without jumps, and their control spaces are assumed to be bounded. 
As stated above, in this paper, instead of seeking for a (possibly discontinuous) solution of (\ref{eq_0_4}) through (\ref{eq_0_6_2_1_2_1_1}), we obtain a  \emph{continuous solution} of (\ref{eq_0_4}) via its equivalent zero-level set representation in (\ref{eq_0_5}). Viability theory for deterministic and stochastic systems could be viewed as an alternative approach to solve state-constrained problems \cite{Aubin_book_1991, Aubin_SAA_1998, Aubin_book, Buckdahn_Existence_1998}, and its extension to jump diffusion models was studied in \cite{Peng_Acta_2008, Zhu_ACTA_2016}. However, they focus only on the viability property of state constraints (without optimizing the objective functional), their control spaces are bounded, and some additional technical assumptions (e.g. see \cite[(H.3)]{Peng_Acta_2008}) are essentially required. 

The state-constrained problem via the backward reachability approach was first studied in \cite{Bokanowski_SICON_2016}.    The model used in \cite{Bokanowski_SICON_2016} is the SDE driven by Brownian motion without jumps, which is a special case of (\ref{eq_0_2}). Moreover, the HJB equation in \cite{Bokanowski_SICON_2016} is the local equation, which is also a special case of (\ref{eq_0_6}) without the nonlocal integral term (the second line of (\ref{eq_0_6})). The aim of this paper is to generalize the results in \cite{Bokanowski_SICON_2016} to the case of jump diffusion systems. As mentioned below, it turns out that these generalizations are not straightforward due to the jump diffusion part in (\ref{eq_0_1}) and the nonlocal operator in the HJB equation (\ref{eq_0_6}). 


Our first main result given in (\ref{eq_0_5}) is obtained based on the stochastic target theory and the approach developed in \cite{Bokanowski_SICON_2016}.  In particular, using the equivalence relationship between stochastic optimal control and stochastic target problems, established in \cite{Bouchard_SCL_2012}, we show (\ref{eq_0_5}), where $\mathcal{R}_t^\Gamma$ is the backward reachable set with the state constraint given by
\begin{align}
\label{eq_0_6_1_1}
\mathcal{R}_{t}^{\Gamma} & := \{ (a,b) \in \mathbb{R}^n \times \mathbb{R}~|~\exists (u,\alpha,\beta) \in \mathcal{U}_{t,T} \times \mathcal{A}_{t,T} \times \mathcal{B}_{t,T}~\text{such that} \\
&\qquad  y_{T;t,a,b}^{u,\alpha,\beta} \geq m(x_T^{t,a;u}),~ \text{$\mathbb{P}$-a.s. and} ~ x_s^{t,a;u} \in \Gamma,~ \forall s \in [t,T],~ \text{$\mathbb{P}$-a.s.} \},	\nonumber
\end{align}
with $(y_{s;t,a,b}^{u,\alpha,\beta})_{s \in [t,T]}$ being an auxiliary state process controlled by additional control processes $(\alpha,\beta) \in \mathcal{A}_{t,T} \times \mathcal{B}_{t,T}$ that take values from  unbounded control spaces. 
Here, the main technical tool to show the equivalence in (\ref{eq_0_5}) using (\ref{eq_0_6_1_1}) is the martingale representation theorem for general L\'evy processes, by which additional (unbounded) controls $(\alpha,\beta) \in \mathcal{A}_{t,T} \times \mathcal{B}_{t,T}$ are induced. It should be mentioned that \cite{Bokanowski_SICON_2016} also used the result of \cite{Bouchard_SCL_2012} (where only $(u,\alpha) \in \mathcal{U}_{t,T} \times \mathcal{A}_{t,T}$ appeared in (\ref{eq_0_6_1_1})), and we extend it to the case of jump diffusion models.   

The second main result is to show that the auxiliary value function $W$ is a unique continuous viscosity solution of the HJB equation in (\ref{eq_0_6}). The proof for the existence that $W$ is a viscosity solution requires the dynamic programming principle and the application of It\^o's formula of general L\'evy-type stochastic integrals to test functions, which must be different from that in \cite{Bokanowski_SICON_2016}. Furthermore, for the proof of uniqueness, the approach in \cite{Bokanowski_SICON_2016} (that also relies on \cite{Bruder_Hal_2005, Bokanowski_SINA_2009}) cannot be directly applied to our case, since (\ref{eq_0_6}) includes the nonlocal (integral) operator in terms of the singular L\'evy measure $\pi$ induced due to jump diffusions  (the second line of (\ref{eq_0_6})). Note also that in the classical stochastic optimal control problem for jump diffusion systems without state constraints ($\Gamma = \mathbb{R}^n)$, the corresponding control space is assumed to be a compact set \cite{Buckdahn_SPTA_2011, Peng_Acta_2008, Pham_JMSEC_1998, Soner_Jump_SCTA_1998}. Hence, their approaches cannot be adapted to the proof for the uniqueness of the HJB equation in (\ref{eq_0_6}).

Our strategy to prove the uniqueness is to use the equivalent definition of viscosity solutions in terms of (super and sub)jets, where the nonlocal integral operator is decomposed into the singular part with the test function and the nonsingular part with jets (see Lemma \ref{Lemma_8} and \cite{Barles_Poincare_2008}). Then we show the boundedness of the nonlocal singular part with the help of the regularity of test functions. Note that the unboundedness of $\beta \in G^2$ in the nonlocal nonsingular part is resolved with the help of the technical lemma (in Appendix \ref{Appendix_B}) and the proper estimates based on \cite[Proposition 3.7]{Crandall_AMS_1992_Viscosity} after doubling variables. In addition, we convert the second-order local part (the first line of (\ref{eq_0_6})) into the equivalent spectral radius form, by which the unboundedness with respect to $\alpha \in \mathbb{R}^r$ can be handled (see Lemma \ref{Lemma_9}). By combining these steps, we obtain a desired contradiction of the comparison principle, which implies the uniqueness of the viscosity solution for (\ref{eq_0_6}) (see Corollary \ref{Corollary_2}).

The inequality in (\ref{eq_0_6_1_1}) also describes the stochastic target constraint; see \cite{Bouchard_SPTA_2002, Bouchard_SCL_2012, Bouchard_Elie_SICON_2010, Moreau_SICON_2011, Soner_SICON_2002} and the references therein. Specifically, the stochastic target problem for jump diffusion systems considers (see \cite{Bouchard_SPTA_2002, Bouchard_SCL_2012, Moreau_SICON_2011})
 \begin{align}
 \label{eq_0_8}
 \inf\{b \geq 0~|~\exists (u,\alpha,\beta) \in \mathcal{U}_{t,T} \times \mathcal{A}_{t,T} \times \mathcal{B}_{t,T}~\text{such that}~	y_{T;t,a,b}^{u,\alpha,\beta} \geq m(x_T^{t,a;u})\},
 \end{align}
 which does not have state constraints ($\Gamma = \mathbb{R}^n$). It was shown that the value function for (\ref{eq_0_8}) is a discontinuous viscosity solution, and its uniqueness has not been fully addressed particularly for jump diffusion models.\footnote{Note that \cite{Bouchard_SPTA_2002, Bouchard_SCL_2012, Moreau_SICON_2011} studied stochastic target problems for jump diffusion models, in which the comparison result of viscosity solutions was not considered.} On the other hand, this paper shows that with an additional assumption (see (ii) of Assumption \ref{Assumption_2} and Theorem \ref{Theorem_1}), $V$ is expressed as the zero-level subset of $W$ in (\ref{eq_0_5}), where $W$ is a unique continuous viscosity solution of (\ref{eq_0_6}). Hence, our approach provides an explicit way to characterize $V$ as a continuous solution of (\ref{eq_0_6}) even if $V$ is discontinuous, which can be obtained easily using various numerical computation schemes. Note also that for SDEs with Brownian motion (no jumps), \cite{Soner_CPDE_2002} showed the (sub)level-set characterization of (\ref{eq_0_8}) under the bounded control set, in which the uniqueness of viscosity solutions was not addressed. We mention that various level-set approaches for characterization of reachable sets in  (deterministic and stochastic) control problems can be found in \cite{Assellaou_DCDS_2015, Margellos_TAC_2011, Mitchell_TAC_2005, Esfahani_Automatica_2016}.

The rest of the paper is organized as follows. The notation and the precise problem statement are given in Section \ref{Section_2}. In Section \ref{Section_3}, using the theory of stochastic target problems, we obtain the equivalent representation of (\ref{eq_0_4}) given in (\ref{eq_0_5}). In Section \ref{Section_4}, we show that the auxiliary value function $W$ is the continuous viscosity solution of the HJB equation in (\ref{eq_0_6}). The uniqueness of the viscosity solution for (\ref{eq_0_6}) is presented in Section \ref{Section_5}, and its proof is provided in Section \ref{Section_6}. Three appendices include some technical results, which are required to prove the main results of the paper.

\section{Notation and Problem Statement} \label{Section_2}
In this section, we first give the notation used in the paper. We then provide the precise problem formulation.

\subsection{Notation}
Let $\mathbb{R}^n$ be the $n$-dimensional Euclidean space. For $x,y \in \mathbb{R}^n$, $x^\top$ denotes the transpose of $x$, $\langle x,y \rangle$ is the inner product, and $|x| := \langle x, x \rangle^{1/2}$. Let $\mathbb{S}^n$ be the set of $n \times n$ symmetric matrices. Let $\Tr(A)$ be the trace operator for a square matrix $A \in \mathbb{R}^{n \times n}$. Let $\|\cdot\|_{F}$ be the Frobenius norm, i.e., $\|A\|_{F} := \Tr(AA^\top)^{1/2}$ for $A \in \mathbb{R}^{n \times m}$. Let $I_n$ be an $n \times n$ identity matrix. In various places of the paper, an exact value of a positive constant $C$ can vary from line to line, which mainly depends on the coefficients in Assumptions \ref{Assumption_1},  \ref{Assumption_2} and \ref{Assumption_3}, terminal time $T$, and the initial condition, but independent to a specific choice of control.

Let $(\Omega, \mathcal{F}, \mathbb{P})$ be a complete probability space with the natural filtration $\mathbb{F} :=\{\mathcal{F}_s,~ 0 \leq s \leq t\}$ generated by the following two mutually independent stochastic processes and augmented by all the $\mathbb{P}$-null sets in $\mathcal{F}$:
\begin{itemize}
\item an $r$-dimensional standard Brownian motion $B$ defined on $[0,T]$;
\item an $E$-marked right continuous Poisson random measure (process) $N$ defined on $E  \times [0,T]$, where $E := \bar{E} \setminus \{0\}$ with $\bar{E} \subset \mathbb{R}^{l}$ is a Borel subset of $\mathbb{R}^{l}$ equipped with its Borel $\sigma$-field $\mathcal{B}(E)$. The intensity measure of $N$ is denoted by $\hat{\pi}(\dd e, \dd t) := \pi(\dd e) \dd t$, satisfying $\pi(E) < \infty$, where $\{\tilde{N}(A,(0,t]) := (N-\hat{\pi})(A,(0,t])\}_{t \in (0,T]}$ is an associated compensated $\mathcal{F}_t$-martingale random (Poisson) measure of $N$ for any $A \in \mathcal{B}(E)$. Here, $\pi$ is an $\sigma$-finite L\'evy measure on $(E,\mathcal{B}(E))$, which satisfies $\int_{E} (1 \wedge  |e|^2) \pi (\dd e) < \infty$.
\end{itemize}

We introduce the following spaces:
\begin{itemize}
\item $L^p(\Omega,\mathcal{F}_t;\mathbb{R}^n)$, $t \in [0,T]$, $p \geq 1$:  the space of $\mathcal{F}_t$-measurable $\mathbb{R}^n$-valued random vectors, satisfying $\|x\|_{L^p} := \mathbb{E}[|x|^p] < \infty$.
\item $\mathcal{L}_{\mathbb{F}}^p(t,T;\mathbb{R}^n)$, $t \in [0,T]$, $p \geq 1$: the space of $\mathbb{F}$-predictable $\mathbb{R}^n$-valued random processes, satisfying $\|x\|_{\mathcal{L}_{\mathcal{F}}^p} := \mathbb{E}[\int_t^T |x_s|^p \dd s ]^{\frac{1}{p}} < \infty$.
\item $G^2(E,\mathcal{B}(E),\pi;\mathbb{R}^n)$:  the space of square integrable functions such that for $k \in G^2(E,\mathcal{B}(E),\pi;\mathbb{R}^n)$, $k:E \rightarrow \mathbb{R}^n$ satisfies $\|k\|_{G^2} := (\int_{E} |k(e)|^2 \pi(\dd e))^{\frac{1}{2}} < \infty$, where $\pi$ is an $\sigma$-finite L\'evy measure on $(E,\mathcal{B}(E))$. $G^2(E,\mathcal{B}(E),\pi;\mathbb{R}^n)$ is a Hilbert space \cite[page 9]{Applebaum_book}.
\item $\mathcal{G}^2_\mathbb{F}(t,T,\pi;\mathbb{R}^n)$, $t \in [0,T]$: the space of stochastic processes such that for $k \in \mathcal{G}^2_\mathbb{F}(t,T,\pi;\mathbb{R}^n)$, $k:\Omega \times [t,T] \times E \rightarrow \mathbb{R}^n$ is an $\mathcal{P} \times \mathcal{B}(E)$-measurable $\mathbb{R}^n$-valued predictable process satisfying $\|k\|_{\mathcal{G}^2_{\mathbb{F}}} := \mathbb{E}[\int_t^T \int_{E} |k_s(e)|^2 \pi (\dd e) \dd s]^{\frac{1}{2}} < \infty$, where $\mathcal{P}$ denotes the $\sigma$-algebra of $\mathcal{F}_t$-predictable subsets of $\Omega \times [0,T]$. Note that $\mathcal{G}^2_\mathbb{F}(t,T,\pi;\mathbb{R}^n)$ is a Hilbert space \cite[Lemma 4.1.3]{Applebaum_book}.
\item $C([0,T] \times \mathbb{R}^n)$: the set of $\mathbb{R}$-valued continuous functions on $[0,T] \times \mathbb{R}^n$.
\item $C_p([0,T] \times \mathbb{R}^n)$, $p \geq 1$: the set of $\mathbb{R}$-valued  continuous functions such that $f \in C_p([0,T] \times \mathbb{R}^n)$ holds $|f(t,x)| \leq C (1 + |x|^p)$. 
\item $C_b^{l,r}([0,T] \times \mathbb{R}^n)$ $l,r \geq 1$: the set of $\mathbb{R}$-valued continuous functions on $[0,T] \times \mathbb{R}^n$ such that for $f \in C^{l,r}([0,T] \times \mathbb{R}^n)$, $\partial_t^{l} f$ and $D^r f$ exist, and are continuous and uniformly bounded, where $\partial_t^{l} f$ is the $l$th-order partial derivative of $f$ with respect to $t \in [0,T]$ and $D^r f$ is the $r$th-order derivative of $f$ in $x \in \mathbb{R}^n$.
\end{itemize}

\subsection{Problem Statement}\label{Section_2_1}
We consider the following stochastic differential equation (SDE) driven by both $B$ and $\tilde{N}$:
\begin{align}
\label{eq_1}
\begin{cases}
\dd x_{s}^{t,a;u} = f(s,x_{s}^{t,a;u},u_s)\dd s + \sigma(s,x_{s}^{t,a;u},u_s) \dd B_s	\\
\qquad \qquad + \int_{E} \chi(s,x_{s-}^{t,a;u},u_s,e) \tilde{N}(\dd e, \dd s), ~ s \in (t,T]\\
x_{t}^{t,a;u} = a,
\end{cases}
\end{align}
where $x \in \mathbb{R}^n$ is the state
 and $u \in U$ is the control with $U$ being the control space, which is a compact subset of $\mathbb{R}^m$. We impose the following assumption:

\begin{assumption}\label{Assumption_1}
$f: [0,T] \times \mathbb{R}^n \times U \rightarrow \mathbb{R}^n$, $\sigma: [0,T] \times \mathbb{R}^n \times U \rightarrow \mathbb{R}^{n \times r}$ and $\chi: [0,T] \times \mathbb{R}^n \times U \times E \rightarrow \mathbb{R}^n$ are continuous in $(t,x,u) \in [0,T] \times \mathbb{R}^n \times U$, and hold the following conditions with the constant $L > 0$: for $x,x^\prime \in \mathbb{R}^n$,
	\begin{align*}
	|f(t,x,u) - f(t,x^\prime,u)| + |\sigma(t,x,u) - \sigma(t,x^\prime,u)| & \leq L |x-x^\prime| \\
	\| \chi(t,x,u,e) -  \chi(t,x^\prime,u,e)\|_{G^2} & \leq L |x-x^\prime| \\
	|f(t,x,u)| + |\sigma(t,x,u)| + \| \chi(t,x,u,e)\|_{G^2} & \leq L (1 + |x|).
	\end{align*}	
\end{assumption}

The set of admissible controls is denoted by $\mathcal{U}_{t,T} := \mathcal{L}_{\mathbb{F}}^2(t,T;U)$. Then under Assumption \ref{Assumption_1}, we have the following estimates for (\ref{eq_1}). Since we could not find these estimates in the existing literature, a complete proof is given in Appendix \ref{Appendix_A}.

\begin{lemma}\label{Lemma_1}
Suppose that Assumption \ref{Assumption_1} holds. Then the following results hold:
\begin{enumerate}[(i)]
	\item For any $a \in \mathbb{R}^n$ and $u \in \mathcal{U}_{t,T}$, there is a unique $\mathbb{F}$-adapted c\`adl\`ag process such that (\ref{eq_1}) holds;
	\item For any $a,a^\prime \in \mathbb{R}^n$, $u \in \mathcal{U}_{t,T}$, and $t,t^\prime \in [0,T]$ with $t \leq t^\prime$, there exists a constant $C>0$ such that
	\begin{align}	
	\label{eq_1_1}
		\mathbb{E} \Bigl [ \sup_{s \in [t,T]} |x_s^{t,a;u}|^2 \Bigr ] & \leq C(1+|a|^2) \\
	\label{eq_1_2}
		\mathbb{E} \Bigl [\sup_{s \in [t,T]} |x_s^{t,a;u} - x_{s}^{t,a^\prime;u}|^2 \Bigr ] & \leq  C|a-a^\prime|^2 \\
	\label{eq_1_3}
		\mathbb{E} \Bigl [ \sup_{s \in [t^\prime,T]} |x_{s}^{t,a;u} - x_{s}^{t^\prime,a;u}|^2 \Bigr ] & \leq  C(1+|a|^2)|t^\prime-t|.
	\end{align}
\end{enumerate}
\end{lemma}

The objective functional is given by
\begin{align}
\label{eq_2}
J(t,a;u) = \mathbb{E} \Bigl [ \int_t^T l(s,x_{s}^{t,a;u},u_s) \dd s + m(x_{T}^{t,a;u})	\Bigr ].
\end{align}
Let $\Gamma \subset \mathbb{R}^n$ be the nonempty and closed set, which captures the state constraint. Then the state-constrained stochastic control problem for jump diffusion systems considered in this paper is as follows:
\begin{align*}
& \inf_{u \in \mathcal{U}_{t,T}}	J(t,a;u)\\
& \text{subject to (\ref{eq_1}) and } 	x_{s}^{t,a;u} \in \Gamma,~ \text{$\forall s \in [t,T]$, $\mathbb{P}$-a.s.}
\end{align*}
We introduce the value function for the above problem:
\begin{align}
\label{eq_4}
V(t,a) := \inf_{u \in \mathcal{U}_{t,T}} \{ J(t,a;u)~|~ x_{s}^{t,a;u} \in \Gamma,~ \text{$\mathbb{P}$-a.s.,}~ \forall s \in [t,T] \}.
\end{align}
Th following assumptions are imposed for (\ref{eq_2}), under which (\ref{eq_4}) is well defined.

\begin{assumption}\label{Assumption_2}
\begin{enumerate}[(i)]
	\item $l: [0,T] \times \mathbb{R}^n \times U \rightarrow \mathbb{R}$ and $m:\mathbb{R}^n \rightarrow \mathbb{R}$ are continuous in $(t,x,u) \in [0,T] \times \mathbb{R}^n \times U$. $l$ and $m$ satisfy the following conditions with the constant $L > 0$: for $x,x^\prime \in \mathbb{R}^n$, 
	\begin{align*}
	 |l(t,x,u) - l(t,x^\prime,u)| + |m(x) - m(x^\prime)| & \leq L |x-x^\prime| \\
	|l(t,x,u)| + |m(x)| & \leq L(1 + |x|);
	\end{align*}
	\item $l$ and $m$ are nonnegative functions, i.e., $l,m \geq 0$.
\end{enumerate}	
\end{assumption}

\begin{remark}\label{Remark_1}
In view of (ii) of Assumption \ref{Assumption_2}, $J(t,a;u) \geq 0$ for any $(t,a,u) \in [0,T] \times \mathbb{R}^n \times \mathcal{U}_{t,T}$, which implies that $V(t,a) \geq 0$ for $(t,a) \in [0,T] \times \mathbb{R}^n$.
\end{remark}

\section{Equivalent Stochastic Target Problem} \label{Section_3}
In this section, we convert the original problem in (\ref{eq_4}) into the stochastic target problem for jump diffusion systems with state constraints. Then we show that (\ref{eq_4}) can be characterized by the backward reachable set of the stochastic target problem, which is equivalent to the zero-level set of the auxiliary value function.

\subsection{Equivalent Stochastic Target Problem via Backward Reachability Approach}
We first introduce an auxiliary SDE associated with the objective functional in (\ref{eq_2}):
\begin{align}
\label{eq_5}
\begin{cases}
\dd y_{s; t,a,b}^{u,\alpha,\beta} = - 	l(s,x_{s}^{t,a;u},u_s) \dd s + \alpha_s^\top \dd B_s + \int_{E} \beta_s(e)  \tilde{N}(\dd e, \dd s),~ s \in (t,T] \\
y_{t; t,a,b}^{u,\alpha,\beta} = b,
\end{cases} 
\end{align}
where $b \in \mathbb{R}$, $u \in \mathcal{U}_{t,T}$, $\alpha\in \mathcal{L}_{\mathbb{F}}^2(t,T;\mathbb{R}^r) =: \mathcal{A}_{t,T}$ and $\beta \in \mathcal{G}_{\mathbb{F}}^2(t,T,\pi;\mathbb{R}) =: \mathcal{B}_{t,T}$. The following estimates hold for  (\ref{eq_5}). The proof is similar to that for Lemma \ref{Lemma_1}. 

\begin{lemma}\label{Lemma_2}
Suppose that Assumptions \ref{Assumption_1} and \ref{Assumption_2} hold. Then:
\begin{enumerate}[(i)]
	\item For any $(u,\alpha,\beta) \in \mathcal{U}_{t,T} \times \mathcal{A}_{t,T} \times \mathcal{B}_{t,T}$ and $(a,b) \in \mathbb{R}^{n+1}$, there is a unique $\mathbb{F}$-adapted c\`adl\`ag process such that (\ref{eq_5}) holds;
	\item For any $(u,\alpha,\beta) \in \mathcal{U}_{t,T} \times \mathcal{A}_{t,T} \times \mathcal{B}_{t,T}$, $(a,b) \in \mathbb{R}^{n+1}$, $(a^\prime,b^\prime) \in \mathbb{R}^{n+1}$, and $t,t^\prime \in [0,T]$ with $t \leq t^\prime$, there exists a constant $C> 0$ such that
	\begin{align*}
	\mathbb{E} \Bigl [ \sup_{s \in [t,T]} | y_{s;t,a,b}^{u,\alpha,\beta} - y_{s;t,a^\prime,b^\prime}^{u,\alpha,\beta} |^2 \Bigr ] & \leq C (|a-a^\prime|^2 + |b-b^\prime|^2) \\
		\mathbb{E} \Bigl [ \sup_{s \in [t^\prime,T]} | y_{s;t,a,b}^{u,\alpha,\beta} - y_{s;t^\prime,a,b}^{u,\alpha,\beta} |^2 \Bigr ] & \leq C (1+|a|^2+|b|^2)|t^\prime - t|.
	\end{align*}
\end{enumerate}	
\end{lemma}

\begin{remark}\label{Remark_2}
We can impose explicit bounds for additional control variables $(\alpha,\beta) \in \mathcal{A}_{t,T} \times \mathcal{B}_{t,T}$. In particular, let $\tilde{J}(t,a;u) := \int_t^T l(s,x_{s}^{t,a;u},u_s) \dd s + m(x_{T}^{t,a;u})$. Since $\tilde{J} \in L^2(\Omega,\mathcal{F}_T;\mathbb{R})$, in view of the martingale representation theorem \cite[Theorem 5.3.5]{Applebaum_book}, there exist unique $(\alpha,\beta) \in \mathcal{A}_{t,T} \times \mathcal{B}_{t,T}$ such that
\begin{align*}
\tilde{J}(t,a;u) = J(t,a;u) + \int_t^T \alpha_s^\top \dd B_s + \int_t^T \int_{E} \beta_s(e) \tilde{N}(\dd e, \dd s),
\end{align*}
which implies
\begin{align*}
	& \int_t^T \alpha_s^\top \dd B_s + \int_t^T \int_{E} \beta_s(e) \tilde{N}(\dd e, \dd s) \\
	&  = \int_t^T l(s,x_{s}^{t,a;u},u_s) \dd s + m(x_{T}^{t,a;u}) - \mathbb{E} \Bigl [ \int_t^T l(s,x_{s}^{t,a;u},u_s) \dd s + m(x_{T}^{t,a;u}) \Bigr ].
\end{align*}
Then from (i) of Assumption \ref{Assumption_2}, the estimates in (ii) of Lemma \ref{Lemma_1}, and the fact that $\tilde{N}$ and $B$ are mutually independent, we have
\begin{align*}
\|\alpha\|_{\mathcal{L}^2_{\mathbb{F}}}^2 \leq C (1+|a|^2),~ \|\beta\|_{\mathcal{G}_{\mathbb{F}}^2}^2 \leq C (1+|a|^2).
\end{align*}
Hence, without loss of generality, we may restrict uniform bounded controls of $(\alpha,\beta)$ in $\mathcal{L}^2_{\mathbb{F}}$ and $\mathcal{G}_{\mathbb{F}}^2$ senses.
\end{remark}

For any function $m:\mathbb{R}^n \rightarrow \mathbb{R}$, let us define the \emph{epigraph} of $m$:
\begin{align*}
\mathcal{E}(m) := \{(x,y) \in \mathbb{R}^{n} \times \mathbb{R}~|~ y \geq m(x) \}	.
\end{align*}
Then we have the following equivalent expression of the value function in (\ref{eq_4}) in terms of the stochastic target problem with state constraints. Below, we drop ${t,T}$ in $\mathcal{U}_{t,T}$, $\mathcal{A}_{t,T}$ and $\mathcal{B}_{t,T}$ to simplify the notation.

\begin{lemma}\label{Lemma_3}
Assume that Assumptions \ref{Assumption_1} and \ref{Assumption_2} hold. Then:
\begin{align}
\label{eq_9}
V(t,a) & = \inf \{ b \geq 0~|~ \exists (u,\alpha,\beta ) \in \mathcal{U} \times \mathcal{A} \times \mathcal{B}~\text{such that} \\
 	& \qquad \quad (x_T^{t,a;u},y_{T;t,a,b}^{u,\alpha,\beta}) \in \mathcal{E}(m),~ \text{$\mathbb{P}$-a.s.}~ \text{and}~x_{s}^{t,a;u} \in \Gamma,~ \forall s \in [t,T],~ \text{$\mathbb{P}$-a.s.} \}  	 \nonumber
\end{align}	
\end{lemma}
\begin{remark}\label{Remark_1_1_1_1}
	We note that (\ref{eq_9}) is the stochastic target problem for jump diffusion systems with state constraints; see \cite{Bouchard_SPTA_2002, Bouchard_SCL_2012, Bouchard_Elie_SICON_2010, Moreau_SICON_2011, Soner_SICON_2002}.
\end{remark}

\begin{proof}[Proof of Lemma \ref{Lemma_3}]
It is easy to see that 
	\begin{align}
	\label{eq_10_1_1}
V(t,a) = & \inf \{ b \geq 0~|~ \exists u \in \mathcal{U}~\text{such that}~  b \geq J(t,a;u) \\
 	& \qquad \qquad \text{and}~ x_{s}^{t,a;u} \in \Gamma,~ \forall s \in [t,T], ~ \text{$\mathbb{P}$-a.s.} \} \nonumber
\end{align}
As discussed in \cite{Bouchard_SCL_2012} and \cite{Bokanowski_SICON_2016}, we consider the following two statements: for $b \geq 0$,
\begin{enumerate}[(a)]
\setlength{\itemindent}{1.2em}
\item There exists $u \in \mathcal{U}$ such that $b \geq J(t,a;u)$ and $x_{s}^{t,a;u} \in \Gamma$ for $ s \in [t,T]$, $\mathbb{P}$-a.s.;
\item There exist $(u,\alpha,\beta ) \in \mathcal{U} \times \mathcal{A} \times \mathcal{B}$ such that $y_{T; t,a,b}^{u,\alpha,\beta} \geq m(x_{T}^{t,a;u})$, $\mathbb{P}$-a.s. and $x_{s}^{t,a;u} \in \Gamma$ for $s \in [t,T]$, $\mathbb{P}$-a.s.
\end{enumerate}
Note that (a) corresponds to (\ref{eq_10_1_1}), while (\ref{eq_9}) is equivalent to (b). Then it is necessary to show the equivalence between (a) and (b).

First, from (b), there exist $(u,\alpha,\beta ) \in \mathcal{U} \times \mathcal{A} \times \mathcal{B}$ such that $y_{T; t,a,b}^{u,\alpha,\beta} \geq m(x_{T}^{t,a;u})$ and by (\ref{eq_5}),
\begin{align}
\label{eq_11}
b \geq   m(x_{T}^{t,a;u}) + \int_t^T l(s,x_{s}^{t,a;u},u_s) \dd s  - \int_t^T \alpha_s^\top \dd B_s - \int_t^T \int_{E} \beta_s(e)  \tilde{N}(\dd e, \dd s).	 
\end{align}
Since the stochastic integrals $\int_t^r \alpha_s^\top \dd B_s$ and $\int_t^r \int_{E} \beta_s(e) \tilde{N}(\dd e, \dd s)$ are $\mathcal{F}_r$-martingales, by taking the expectation in (\ref{eq_11}), we get $b \geq J(t,a;u)$. Hence, (b) implies (a).

On the other hand, let $\tilde{J}(t,a;u) := \int_t^T l(s,x_{s}^{t,a;u},u_s) \dd s + m(x_{T}^{t,a;u})$. Since $\tilde{J} \in L^2(\Omega,\mathcal{F}_T;\mathbb{R})$, in view of the martingale representation theorem \cite[Theorem 5.3.5]{Applebaum_book}, there exist unique $(\tilde{\alpha},\tilde{\beta} ) \in  \mathcal{A} \times \mathcal{B}$ such that
\begin{align*}
	\tilde{J}(t,a;u) = J(t,a;u) + \int_t^T \tilde{\alpha}_s^\top \dd B_s + \int_t^T \int_{E} \tilde{\beta}_s(e) \tilde{N}(\dd e, \dd s).
\end{align*}
Then from (a), for $b \geq 0$,
\begin{align*}
b \geq & J(t,a;u)  \\
& = 	\int_t^T l(s,x_{s}^{t,a;u},u_s) \dd s + m(x_{T}^{t,a;u})  - \int_t^T \tilde{\alpha}_s^\top \dd B_s - \int_t^T \int_{E} \tilde{\beta}_s (e) \tilde{N}(\dd e, \dd s),
\end{align*}
which, together with (\ref{eq_5}), shows that $y_{T; t,a,b}^{u,\tilde{\alpha},\tilde{\beta}} \geq m(x_{T}^{t,a;u})$. Hence, (a) implies (b). This completes the proof.
\end{proof}

We now introduce the \emph{backward reachable set}
\begin{align}
\label{eq_9_1}
\mathcal{R}_{t}^{\Gamma} & := \{ (a,b) \in \mathbb{R}^n \times \mathbb{R}~|~\exists (u,\alpha,\beta) \in \mathcal{U} \times \mathcal{A} \times \mathcal{B}~\text{such that} \\
&\qquad  (x_T^{t,a;u},y_{T;t,a,b}^{u,\alpha,\beta}) \in \mathcal{E}(m),~ \text{$\mathbb{P}$-a.s. and} ~ x_s^{t,a;u} \in \Gamma,~ \forall s \in [t,T],~ \text{$\mathbb{P}$-a.s.} \}	 \nonumber
\end{align}
Clearly, based on Lemma \ref{Lemma_3}, we have the following result:
\begin{theorem}\label{Corollary_1}
Assume that Assumptions \ref{Assumption_1} and \ref{Assumption_2} hold. For any $(t,a) \in [0,T] \times \mathbb{R}^n$, 
\begin{align}
\label{eq_3_5_1_2_1_2}
V(t,a) = \inf \{ b \geq 0~|~ (a,b) \in \mathcal{R}_{t}^{\Gamma} \}.
\end{align}
\end{theorem}

\begin{remark}\label{Remark_3}
From Theorem \ref{Corollary_1}, we observe that the value function in (\ref{eq_4}) can be characterized by the backward reachable set $\mathcal{R}_{t}^{\Gamma}$. In the next subsection, we focus on an explicit characterization of $\mathcal{R}_{t}^{\Gamma}$ as the zero-level set of the value function for the unconstrained auxiliary stochastic control problem.  	
\end{remark}

\subsection{Characterization of Backward Reachable Set}

Let
\begin{align*}
\bar{J}(t,a,b;u,\alpha,\beta) = \mathbb{E} \Bigl [ \max \{ m(x_T^{t,a;u}) - y_{T;t,a,b}^{u,\alpha,\beta}, 0 \}  + \int_t^T d(x_s^{t,a;u},\Gamma) \dd s \Bigr ],
\end{align*}
where we introduce the following \emph{distance function} on $\mathbb{R}^n$ to $\mathbb{R}^+$:
\begin{align*}
d(x,\Gamma) = 0~ \text{if and only if}~ x \in \Gamma.	
\end{align*}
Then the auxiliary value function $W: [0,T] \times \mathbb{R}^n \times \mathbb{R} \rightarrow \mathbb{R}$ can be defined as follows:
\begin{align}
\label{eq_10}
W(t,a,b) & :=  \mathop{\inf_{u \in \mathcal{U}}}_{\alpha \in \mathcal{A},\beta \in \mathcal{B}} \bar{J}(t,a,b;u,\alpha,\beta),~\text{subject to (\ref{eq_1}) and (\ref{eq_5}).}
\end{align}
Note that (\ref{eq_10}) does not have any state constraints. 

\begin{assumption}\label{Assumption_3}
$d(x,\Gamma)$ is Lipschitz continuous in $x$ with the Lipschitz constant $L$ and satisfies the linear growth condition in $x$.
\end{assumption}

\begin{remark}\label{Remark_4}
Examples of $d(x,\Gamma$) are $d(x,\Gamma)	= \inf_{y \in \Gamma} | x - y|$ and $d(x,\Gamma) = x - x\mathds{1}_{x \succeq 0}$, where $\mathds{1}$ is an indicator function and $\succeq$ is the componentwise inequality. Clearly, they hold Assumption \ref{Assumption_3}.
\end{remark}

The following theorem shows the equivalent expression of $V$ in terms of the zero-level set of $W$.

\begin{theorem}\label{Theorem_1}
Suppose that Assumptions \ref{Assumption_1}, \ref{Assumption_2} and \ref{Assumption_3} hold and that there exists an optimal control such that it attains the minimum of the auxiliary optimal control problem in (\ref{eq_10}). Then:
\begin{enumerate}[(i)]
	\item The reachable set can be obtained by 
	\begin{align*}
	\mathcal{R}_t^\Gamma = \{(a,b) \in \mathbb{R}^n \times \mathbb{R}~|~ W(t,a,b) = 0 \},~ \forall t \in [0,T];
	\end{align*}
	\item The value function $V$ in (\ref{eq_4}) can be characterized by the zero-level set of $W$: for $(t,a) \in [0,T] \times \mathbb{R}^n$, 
	\begin{align}
	\label{eq_3_7_5_3_2_1_1_1}
		V(t,a) = \inf \{b \geq 0~| (a,b) \in \mathcal{R}_t^\Gamma \} = \inf \{b \geq 0~|~ W(t,a,b) = 0, a \in \mathbb{R}^n \}.
	\end{align}
\end{enumerate}
\end{theorem}

\begin{remark}\label{Remark_5}
In Sections \ref{Section_4} and \ref{Section_5}, we show that $W$ is a unique viscosity solution of the associated Hamilton-Jacobi-Bellman (HJB) equation. Hence, from Theorem \ref{Theorem_1} (particularly (\ref{eq_3_7_5_3_2_1_1_1})), the value function of the state-constrained problem $V$ in (\ref{eq_4}) can be obtained by solving the HJB equation of $W$.
\end{remark}

\begin{proof}[Proof of Theorem \ref{Theorem_1}]
	From (\ref{eq_3_5_1_2_1_2}) in Theorem \ref{Corollary_1}, we see that (ii) follows from (i). Hence, we prove (i). Recall $\mathcal{R}_t^\Gamma$ defined in (\ref{eq_9_1}):
	\begin{align*}
\mathcal{R}_{t}^{\Gamma} & := \{ (a,b) \in \mathbb{R}^n \times \mathbb{R}~|~\exists (u,\alpha,\beta) \in \mathcal{U} \times \mathcal{A} \times \mathcal{B}~\text{such that} \\
&\qquad  (x_T^{t,a;u},y_{T;t,a,b}^{u,\alpha,\beta}) \in \mathcal{E}(m),~ \text{$\mathbb{P}$-a.s. and} ~ x_s^{t,a;u} \in \Gamma,~ \forall s \in [t,T],~ \text{$\mathbb{P}$-a.s.} \},	 \nonumber
\end{align*}
and let $\bar{\mathcal{R}}_t^\Gamma := \{(a,b) \in \mathbb{R}^n \times \mathbb{R}~|~ W(t,a,b) = 0 \}$. 
	We will show that $\mathcal{R}_t^\Gamma \subseteq \bar{\mathcal{R}}_t^\Gamma$ and  $\mathcal{R}_t^\Gamma \supseteq \bar{\mathcal{R}}_t^\Gamma$  for $t \in [0,T]$.

	Fix $(a,b) \in  \mathcal{R}_t^\Gamma$. By definition, there exist $(u,\alpha,\beta) \in \mathcal{U} \times \mathcal{A} \times \mathcal{B}$ such that
	\begin{align*}
		\max \{ m(x_T^{t,a;u}) - y_{T;t,a,b}^{u,\alpha,\beta}, 0 \} = 0~ \text{and}~
		d(x_s^{t,a;u},\Gamma) = 0,~ \forall s \in [t,T],~\text{$\mathbb{P}$-a.s.}
	\end{align*}
	This implies that $W(t,a,b) = 0$ for $t \in [0,T]$; hence, $\mathcal{R}_t^\Gamma \subseteq \bar{\mathcal{R}}_t^\Gamma$ for $t \in [0,T]$.

	Suppose that $(a,b) \in \bar{\mathcal{R}}_t^\Gamma$, i.e., $W(t,a,b) = 0$. Then due to the assumption of the existence of an optimal control given in the statement\footnote{In Appendix \ref{Appendix_C}, we discuss the existence of optimal controls for jump diffusion systems.}, there exist $(\bar{u},\bar{\alpha},\bar{\beta}) \in \mathcal{U} \times \mathcal{A} \times \mathcal{B}$ such that
	\begin{align*}
	W(t,a,b) = \mathbb{E} \Bigl [ \max \{ m(x_T^{t,a;\bar{u}}) - y_{T;t,a,b}^{\bar{u},\bar{\alpha},\bar{\beta}}, 0 \} + \int_t^T d(x_s^{t,a;\bar{u}},\Gamma) \dd s\Bigr ] = 0.	
	\end{align*}
	From the nonnegativity of $d(x,\Gamma)$ in Assumption \ref{Assumption_3}, we can see that $\max \{ m(x_T^{t,a;\bar{u}}) - y_{T;t,a,b}^{\bar{u},\bar{\alpha},\bar{\beta}}, 0 \} + \int_t^T d(x_s^{t,a;\bar{u}},\Gamma) \dd s  < 0$, $\mathbb{P}$-a.s. is not possible. If $\max \{ m(x_T^{t,a;\bar{u}}) - y_{T;t,a,b}^{\bar{u},\bar{\alpha},\bar{\beta}}, 0 \} + \int_t^T d(x_s^{t,a;\bar{u}},\Gamma) \dd s  > 0$, $\mathbb{P}$-a.s., then it contradicts $W(t,a,b) = 0$. Hence, we must have
	\begin{align*}
	\max \{ m(x_T^{t,a;\bar{u}}) - y_{T;t,a,b}^{\bar{u},\bar{\alpha},\bar{\beta}}, 0 \} + \int_t^T d(x_s^{t,a;\bar{u}},\Gamma) \dd s =0,~ \text{$\mathbb{P}$-a.s.},	
	\end{align*}
	which, together with the nonnegativity of $d(x,\Gamma)$, leads to
	\begin{align*}
		(x_T^{t,a;\bar{u}},y_{T;t,a,b}^{\bar{u},\bar{\alpha},\bar{\beta}}) \in \mathcal{E}(m)~ \text{and}~ x_s^{t,a;\bar{u}} \in \Gamma,~ \forall s \in [t,T],~ \text{$\mathbb{P}$-a.s.} 
	\end{align*}
This shows that $\bar{\mathcal{R}}_t^\Gamma \subseteq \mathcal{R}_t^\Gamma$  for $t \in [0,T]$. We complete the proof.
\end{proof}

\subsection{Properties of $W$}
We provide some useful properties of $W$ in (\ref{eq_10}). 

First, based on \cite[Theorem 3.3]{Touzi_Book} and \cite[Theorem 3.3, Chapter 4]{Yong_book}, we state the dynamic programming principle for $W$. This will be used in Section \ref{Section_4} to show the existence of the viscosity solution for the HJB equation.

\begin{proposition}\label{Proposition_1}
Assume that Assumptions \ref{Assumption_1}, \ref{Assumption_2} and \ref{Assumption_3} hold. Then for $(a,b) \in \mathbb{R}^n \times \mathbb{R}$ and $t \in [0,T]$ with $r \in (t,T]$, where $r$ is an $\mathbb{F}$-stopping time, $W$ satisfies the following dynamic programming principle (DPP):
\begin{align*}
W(t,a,b) & = \mathop{\inf_{u \in \mathcal{U}}}_{\alpha \in \mathcal{A},\beta \in \mathcal{B}} \mathbb{E} \Bigl [\int_{t}^{r} d(x_s^{t,a;u},\Gamma) \dd s  +  W(r,x_{r}^{t,a;u}, y_{r;t,a,b}^{u,\alpha,\beta}) \Bigr ].
\end{align*}
\end{proposition}

The following lemma shows the continuity of $W$.
\begin{lemma}\label{Lemma_4}
Suppose that Assumptions \ref{Assumption_1}, \ref{Assumption_2} and \ref{Assumption_3} hold. Then for $t \in [0,T]$, there exists a constant $C>0$ such that
\begin{enumerate}[(i)]
\item $|W(t,a,b)| \leq C(1 + |a|)$ for any $(a,b) \in \mathbb{R}^n \times [0,\infty)$;
\item $W$ is Lipschitz continuous in $\mathbb{R}^n \times \mathbb{R}$,  i.e., for $(a,b) \in \mathbb{R}^n \times \mathbb{R}$ and $(a^\prime,b^\prime) \in \mathbb{R}^n \times \mathbb{R}$,
\begin{align*}
|W(t,a,b) - W(t,a^\prime,b^\prime)| &\leq C(|a - a^\prime| + |b - b^\prime|);
\end{align*}
\item $W$ is continuous in $t \in [0,T]$.
\end{enumerate}
\end{lemma}
\begin{proof}
	In view of the definition of $W$,  when $b \in [0,\infty)$, with $\alpha=0$ and $\beta=0$, 
	\begin{align*}
	W(t,a,b) & \leq \mathop{\inf_{u \in \mathcal{U}}}
	\mathbb{E} \Bigl [ \max \{ m(x_T^{t,a;u}) - y_{T;t,a,b}^{u,\alpha,\beta}, 0 \}  + \int_t^T d(x_s^{t,a;u},\Gamma) \dd s \Bigr ] \\
& \leq \mathbb{E} \Bigl [ \int_t^T l(s,x_{s}^{t,a;u},u_s) \dd s + m(x_{T}^{t,a;u})  + \int_t^T d(x_{s}^{t,a;u}, \Gamma) \dd s \Bigr ],
	\end{align*}
where the second inequality follows from the fact that $l$ and $m$ are nonnegative due to (ii) of Assumption \ref{Assumption_2}. Then the linear growth of $W$ in $a$ in the statement of (i) follows from Assumptions \ref{Assumption_1}, \ref{Assumption_2} and \ref{Assumption_3}, and (ii) of Lemma \ref{Lemma_1}. 

Note that $|\inf f(x) - \inf g(x)| \leq \sup |f(x) - g(x)|$ and $|\sup f(x) - \sup g(x)| \leq \sup |f(x) - g(x)|$. From Assumptions \ref{Assumption_1}, \ref{Assumption_2} and \ref{Assumption_3}, and using H\"older inequality,
\begin{align*}
	& |W(t,a,b) - W(t,a^\prime, b^\prime)| \\
	& \leq C \mathop{\sup_{u \in \mathcal{U}}}_{\alpha \in \mathcal{A},\beta \in \mathcal{B}} \Bigl \{ \mathbb{E} \Bigl [ |x_{T}^{t,a;u}  - x_{T}^{t,a^\prime;u} |^2 \Bigr ]^{\frac{1}{2}} + \mathbb{E} \Bigl [ |y_{T;t,a,b}^{u,\alpha,\beta} - y_{T;t,a^\prime,b^\prime}^{u,\alpha,\beta}|^2 \Bigr ]^{\frac{1}{2}}  \\
	& \qquad \qquad + \mathbb{E} \Bigl [ \int_t^T |x_{s}^{t,a;u} - x_{s}^{t,a^\prime;u} |^2 \dd s \Bigr ]^{\frac{1}{2}} \Bigr \}  \leq C(|a - a^\prime | + |b - b^\prime|).
\end{align*}
Notice that to obtain the last inequality, we have used (ii) of Lemmas \ref{Lemma_1} and \ref{Lemma_2}, the compactness of $U$, and the fact that $(\alpha,\beta)$ can be restricted to uniformly bounded controls in $\mathcal{G}_{\mathbb{F}}^2$ and $\mathcal{L}_{\mathbb{F}}^2$ senses from Remark \ref{Remark_2}. This shows (ii).

For the continuity of $W$ in $t \in [0,T]$ in (iii), let $t,t+\tau \in [0,T]$ with $\tau > 0$. Then by applying the similar technique above and using (ii) of Lemma \ref{Lemma_1}, we have
\begin{align*}
& |W(t+\tau,a,b) - W(t,a,b)| \\
& \leq C \mathop{\sup_{u \in \mathcal{U}}}_{\alpha \in \mathcal{A},\beta \in \mathcal{B}}  \Bigl \{ \mathbb{E} \Bigl [ |x_{T}^{t+\tau,a;u}  - x_{T}^{t,a;u} |^2 \Bigr ]^{\frac{1}{2}} + \mathbb{E} \Bigl [ |y_{T;t+ \tau,a,b}^{u,\alpha,\beta} - y_{T;t,a,b}^{u,\alpha,\beta}|^2 \Bigr ]^{\frac{1}{2}}\\
	& \qquad \qquad  + \mathbb{E} \Bigl [ \int_{t+\tau}^T |x_{s}^{t+\tau,a;u} - x_{s}^{t,a;u} |^2 \dd s \Bigr ]^{\frac{1}{2}} + \mathbb{E} \Bigl [ \int_{t}^{t+\tau} (1 + |x_{s}^{t,a;u}|^2) \dd s \Bigr ]^{\frac{1}{2}} \Bigr \} \\
	& \leq C \Bigl ( \tau^{\frac{1}{2}} + \mathop{\sup_{u \in \mathcal{U}}}_{\alpha \in \mathcal{A},\beta \in \mathcal{B}} \mathbb{E} \Bigl [ |y_{T;t+ \tau,a,b}^{u,\alpha,\beta} - y_{T;t,a,b}^{u,\alpha,\beta}|^2 \Bigr ]^{\frac{1}{2}} \Bigr ).
\end{align*}
From Remark \ref{Remark_2}, we may consider the uniformly bounded controls of $(\alpha,\beta)$ in $\mathcal{G}_{\mathbb{F}}^2$ and $\mathcal{L}_{\mathbb{F}}^2$ senses. Therefore, we apply (ii) of Lemma \ref{Lemma_2} to get 
\begin{align*}	
\limsup_{\tau \downarrow 0} |W(t+\tau,a,b) - W(t,a,b)|  & \leq \limsup_{\tau \downarrow 0} C (1 + |a| + |b|) \tau^{\frac{1}{2}} = 0.
\end{align*}
We complete the proof. 
\end{proof}	

\begin{lemma}\label{Lemma_5}
Suppose that Assumptions \ref{Assumption_1}, \ref{Assumption_2} and \ref{Assumption_3} hold. If $b \leq 0$, then we have
\begin{align*}
	W(t,a,b) = W_0(t,a) - b,
\end{align*}
where $W_0:[0,T] \times \mathbb{R}^n \rightarrow \mathbb{R}$ is the value function of the following problem:
\begin{align*}
W_0(t,a) := \inf_{u \in \mathcal{U}} \Bigl \{ J(t,a;u) + \mathbb{E}  \int_t^T d(x_s^{t,a;u},\Gamma) \dd s  \Bigr \}.	
\end{align*}
\end{lemma}
\begin{proof}
Note that since $b \leq 0$, it follows from the nonnegativity of $l$ and $m$ that
\begin{align*}	
& \mathbb{E} \Bigl [ m(x_T^{t,a;u}) - y_{T;t,a,b}^{u,\alpha,\beta} \Bigr ] \\
& = \mathbb{E} \Bigl [ m(x_T^{t,a;u}) - b + \int_t^T l(s,x_{s}^{t,a;u},u_s) \dd s - \int_t^T \alpha_s^\top \dd B_s - \int_t^T \int_{E} \beta_s(e) \tilde{N}(\dd e, \dd s) \Bigr ] \\
& = \mathbb{E} \Bigl [ m(x_T^{t,a;u}) - b + \int_t^T l(s,x_{s}^{t,a;u},u_s) \dd s  \Bigr ] \geq 0,
\end{align*}
which is due to the fact that $\mathbb{E}[\int_t^T \alpha_s^\top \dd B_s] = 0$ and $\mathbb{E}[\int_t^T \int_{E} \beta_s(e)  \tilde{N}(\dd e, \dd s)] = 0$. 

Hence,
\begin{align*}
\mathbb{E} \Bigl [ \max \{ m(x_T^{t,a;u}) - y_{T;t,a,b}^{u,\alpha,\beta}, 0 \} \Bigr ] = \mathbb{E} \Bigl [ m(x_T^{t,a;u}) - y_{T;t,a,b}^{u,\alpha,\beta} \Bigr ],
\end{align*}
and from the definition of $W$ and $J$,\begin{align*}
 W(t,a,b) & = \mathop{\inf_{u \in \mathcal{U}}}_{\alpha \in \mathcal{A},\beta \in \mathcal{B}} \mathbb{E} \Bigl [	m(x_T^{t,a;u}) - b + \int_t^T l(s,x_{s}^{t,a;u},u_s) \dd s  + \int_t^T d(x_s^{t,a;u},\Gamma) \dd s \Bigr ] \\
& =W_0(t,a) - b.
\end{align*}
This completes the proof.
\end{proof}

Based on (\ref{eq_10}) and Lemma \ref{Lemma_5}, $W$ satisfies the following boundary conditions:

\begin{lemma}\label{Lemma_6}
Suppose that Assumptions \ref{Assumption_1}, \ref{Assumption_2} and \ref{Assumption_3} hold. Then $W$ satisfies the following boundary conditions:
\begin{align*}
\begin{cases}
	W(T,\widehat{a}) = \max \{ m(a) - b \},~  (a,b) \in \mathbb{R}^n \times [0,\infty) \\
	W(t,a,0) = W_0(t,a),~  (t,a) \in [0,T) \times \mathbb{R}^n.	
\end{cases}
\end{align*}
\end{lemma}

\section{The Hamilton-Jacobi-Bellman Equation: Existence of Viscosity Solution} \label{Section_4}

In this section and Sections \ref{Section_5} and \ref{Section_6}, we show that $W$ is a unique continuous viscosity solution of the associated HJB equation.

As seen from (\ref{eq_10}), the auxiliary value function depends on the augmented dynamical system on $\mathbb{R}^{n+1}$. We introduce the following notation:
\begin{align*}
\widehat{f}(t,a,u) & := \begin{bmatrix}
 	f(t,a,u) \\
 	-l(t,a,u)
 \end{bmatrix},~ \widehat{\sigma}(t,a,u,\alpha) := \begin{bmatrix}
 	\sigma(t,a,u) \\
 	\alpha^\top
 \end{bmatrix} \\
 \widehat{\chi}(t,a,u,e,\beta) & := \begin{bmatrix}
 	\chi(t,a,u,e)\\
 	\beta(e)
 \end{bmatrix},~ \widehat{a} = \begin{bmatrix}
 	a \\
 	b
 \end{bmatrix},	
\end{align*}
where $\widehat{\sigma}:[0,T] \times \mathbb{R}^n \times U \times \mathbb{R}^p \rightarrow \mathbb{R}^{ (n+1) \times p}$ and $\widehat{\chi}: [0,T] \times \mathbb{R}^n \times U \times  E \times G^2(E,\mathcal{E},\pi;\mathbb{R}) \rightarrow \mathbb{R}^{n+1}$. Let $\mathcal{O} := [0,T) \times \mathbb{R}^n \times (0,\infty)$, $\bar{\mathcal{O}} := [0,T] \times \mathbb{R}^n \times [0,\infty)$, and $G^2 := G^2(E,\mathcal{B}(E),\pi;\mathbb{R})$.

The HJB equation with the boundary conditions (see Lemma \ref{Lemma_6}) is introduced below, which is the second-order nonlinear integro-partial differential equation (IPDE):
\begin{align}
\label{eq_4_1}
\begin{cases}
- \partial_t W(t,\widehat{a})	+ H(t,\widehat{a},(W,DW,D^2W)(t,\widehat{a})) = 0, ~ (t,\widehat{a}) \in \mathcal{O}\\
	W(T,\widehat{a}) = \max \{ m(a) - b \},~  (a,b) \in \mathbb{R}^n \times [0,\infty) \\
	W(t,a,0) = W_0(t,a),~  (t,a) \in [0,T) \times \mathbb{R}^n,
\end{cases}
\end{align}
where the Hamiltonian $H:\bar{\mathcal{O}} \times \mathbb{R} \times \mathbb{R}^{n+1} \times \mathbb{S}^{n+1} \rightarrow \mathbb{R}$ is defined by
\begin{align*}
& H(t,\widehat{a},W,DW,D^2W) \\
& := \mathop{\sup_{u \in U}}_{\alpha \in \mathbb{R}^{r}, \beta \in G^2} \Bigl \{ - \langle D W (t,\widehat{a}), \widehat{f}(t,a,u) \rangle  - \frac{1}{2} \Tr (\widehat{\sigma} \widehat{\sigma}^\top(t,a,u,\alpha) D^2 W (t,\widehat{a}) ) \\
&\qquad \qquad - \int_{E} \bigl [ W(t,\widehat{a} +  \widehat{\chi} (t,a,u,e,\beta)) - W(t,\widehat{a}) \\
&\qquad \qquad \qquad \qquad - \langle DW(t,\widehat{a}),  \widehat{\chi}(t,a,u,e,\beta) \rangle \bigr ]\pi(\dd e) \Bigr \} - d(a,\Gamma).
\end{align*}

The notion of viscosity solutions for (\ref{eq_4_1}) is given as follows \cite{Barles_SSR_1997, Barles_Poincare_2008, Buckdahn_SPTA_2011, Peng_Acta_2008, Pham_JMSEC_1998, Soner_Jump_SCTA_1998}:

\begin{definition}\label{Definition_1}
	A real-valued function $W \in C(\bar{\mathcal{O}} )$ is said to be a viscosity subsolution (resp. supersolution) of (\ref{eq_4_1}) if   
	\begin{enumerate}[(i)]
		\item $W(T,\widehat{a}) \leq \max\{m(a) - b\}$  (resp. $W(T,\widehat{a}) \geq \max\{m(a) - b\}$) for $(a,b) \in \mathbb{R}^n \times [0,\infty)$ and $W(t,a,0) \leq W_0(t,a)$ (resp. $W(t,a,0) \geq W_0(t,a) $) for $(t,a) \in [0,T) \times \mathbb{R}^n$;
		\item For all test functions $\phi \in C_b^{1,3}(\bar{\mathcal{O}}) \cap C_2(\bar{\mathcal{O}})$, the following inequality holds at the global maximum (resp. minimum) point $(t,\widehat{a})  \in \mathcal{O}$ of $W-\phi$:
	\begin{align*}
	& - \partial_t \phi(t,\widehat{a})  + H(t,\widehat{a},(\phi, D \phi, D^2 \phi)(t,\widehat{a})) \leq  0 \\
	 (\text{resp.}~  & - \partial_t \phi(t,\widehat{a})  + H(t,\widehat{a},(\phi, D \phi, D^2 \phi)(t,\widehat{a})) \geq  0 ).
	\end{align*}
	\end{enumerate}
A real-valued function $W \in C(\bar{\mathcal{O}})$ is said to be a viscosity solution of (\ref{eq_4_1}) if it is both a viscosity subsolution and a viscosity supersolution of (\ref{eq_4_1}).
\end{definition}

The existence of the viscosity solution for (\ref{eq_4_1}) can be stated as follows:

\begin{theorem}\label{Theorem_2}	
Suppose that Assumptions \ref{Assumption_1}, \ref{Assumption_2} and \ref{Assumption_3} hold. Then $W$ defined in (\ref{eq_10}) is the continuous viscosity solution of the HJB equation in (\ref{eq_4_1}).
\end{theorem}

\begin{remark}\label{Remark_7}
In the proof of Theorem \ref{Theorem_2}, the additional growth condition of the test function is not required.
\end{remark}

\begin{proof}[Proof of Theorem \ref{Theorem_2}]
	Let us first prove the subsolution property. In view of Lemma \ref{Lemma_4}, $W \in C([0,T] \times \mathbb{R}^{n+1})$. Also, from Lemma \ref{Lemma_6}, $W$ satisfies (i) of Definition \ref{Definition_1}.
	
	We prove (ii) of Definition \ref{Definition_1}. Let $\phi \in C_b^{1,3}(\bar{\mathcal{O}})$ be the test function such that
	\begin{align*}
	(W-\phi)(t,a,b) = \max_{(\bar{t},\bar{a},\bar{b}) \in \mathcal{O}}	(W-\phi)(\bar{t},\bar{a},\bar{b}),
	\end{align*}
	and without loss of generality, we may assume that $W(t,a,b) = \phi(t,a,b)$. This implies $W(\bar{t},\bar{a},\bar{b}) \leq \phi  (\bar{t},\bar{a},\bar{b})$ for $(\bar{t},\bar{a},\bar{b}) \in \mathcal{O}$ and $(\bar{t},\bar{a},\bar{b}) \neq (t,a,b)$.
	
	By using the DPP in Proposition \ref{Proposition_1} with $t,t+\tau \in [0,T]$ and $\tau > 0$,
	\begin{align*}
	& \phi(t,a,b)  = W(t,a,b)  = \mathop{\inf_{u \in \mathcal{U}}}_{\alpha \in \mathcal{A},\beta \in \mathcal{B}} \mathbb{E} \Bigl [\int_{t}^{t+\tau} d(x_s^{t,a;u},\Gamma) \dd s  +  W(t+\tau,x_{t+\tau}^{t,a;u}, y_{t+\tau;t,a,b}^{u,\alpha,\beta}) \Bigr ],	
	\end{align*}
	which implies
\begin{align*}
\phi(t,a,b) - \mathbb{E} \Bigl [\int_{t}^{t+\tau} d(x_s^{t,a;u},\Gamma) \dd s  +  \phi(t+\tau,x_{t+\tau}^{t,a;u}, y_{t+\tau;t,a,b}^{u,\alpha,\beta}) \Bigr ]   \leq 0.
\end{align*}

By applying It\^o's formula of L\'evy-type stochastic integrals \cite[Theorem 4.4.7]{Applebaum_book}, 
\begin{align*}
& - \mathbb{E} \Bigl [ \int_{t}^{t+\tau} d(x_s^{t,a;u},\Gamma) \dd s + \int_t^{t+\tau} \partial_t \phi (s,x_{s}^{t,a;u}, y_{s;t,a,b}^{u,\alpha,\beta})  \dd s \Bigr ]\\
& - \mathbb{E} \Bigl [ \int_{t}^{t+\tau} \langle D \phi(s,x_{s}^{t,a;u}, y_{s;t,a,b}^{u,\alpha,\beta}), \widehat{f}(s,x_{s}^{t,a;u},u_s)	 \rangle \dd s \Bigr ]\\
& - \frac{1}{2} \mathbb{E} \Bigl [ \int_{t}^{t+\tau} \Tr (\widehat{\sigma}\widehat{\sigma}(s,x_{s}^{t,a;u},u_s,\alpha_s) D^2 \phi (s,x_{s}^{t,a;u}, y_{s;t,a,b}^{u,\alpha,\beta})) \dd s \Bigr ]\\
& - \mathbb{E} \Bigl [ \int_{t}^{t+\tau} \int_{E} \bigl [ \phi(s,x_{s}^{t,a;u} + \chi(s,x_{s}^{t,a;u},u_s,e), y_{s;t,a,b}^{u,\alpha,\beta}  + \beta_s(e)) - \phi(s,x_{s}^{t,a;u}, y_{s;t,a,b}^{u,\alpha,\beta}) \\
&\qquad \qquad \qquad - \langle D \phi(s,x_{s}^{t,a;u}, y_{s;t,a,b}^{u,\alpha,\beta}), \widehat{\chi}(s,x_{s}^{t,a;u},u_s,e,\beta_s(e)) \rangle \bigr ] \pi (\dd e) \dd s \Bigr ]
\leq 0,
\end{align*}
where we have used the fact that the expectation for the stochastic integrals of $B$ and $\tilde{N}$ are zero, since they are $\mathcal{F}_t$-martingales.

Multiplying $\frac{1}{\tau}$ above and then letting $\tau \downarrow 0$, we have
\begin{align*}
& - \partial_t \phi (t,\widehat{a}) + H^\prime (t,\widehat{a},(\phi,D\phi,D^2\phi)(t,\widehat{a}); u,\alpha,\beta) \leq 0,
\end{align*}
where
\begin{align}
\label{eq_4_2}
	& H^\prime (t,\widehat{a},(\phi,D\phi,D^2\phi)(t,\widehat{a}); u,\alpha,\beta) \\
	& :=  - d(a,\Gamma) - \langle D \phi(t,\widehat{a}), \widehat{f}(t,a,u) \rangle  - \frac{1}{2} \Tr (\widehat{\sigma} \widehat{\sigma}^\top(t,a,u,\alpha) D^2 \phi (t,\widehat{a}))  \nonumber \\
&~~~ - \int_{E} \bigl [ \phi(t, a + \chi(t,a,u,e),b + \beta(e))) - \phi(t,\widehat{a})  - \langle D \phi(t,\widehat{a}), \widehat{\chi}(t,a,u,e,\beta) \rangle \bigr ]\pi(\dd e). \nonumber
\end{align}
By taking $\sup$ with respect to $(u,\alpha,\beta) \in U \times \mathbb{R}^p \times G^2$, in view of definition $H$, 
\begin{align}
\label{eq_4_3}
- \partial_t \phi (t,\widehat{a})	+ H(t,\widehat{a},(\phi,D \phi,D^2 \phi)(t,\widehat{a})) \leq 0,
\end{align}
which shows that $W$ is the viscosity subsolution of (\ref{eq_4_1}). 

We now prove, by contradiction, the  supersolution property. It is easy to see that $W$ satisfies the boundary inequalities in (i) of Definition \ref{Definition_1}. 

Suppose that $\phi \in C_b^{1,3}(\bar{\mathcal{O}})$ is the test function satisfying the following property:
	\begin{align*}
	(W-\phi)(t,a,b) = \min_{(\bar{t},\bar{a}, \bar{b}) \in \mathcal{O} }	(W-\phi)(\bar{t},\bar{a},\bar{b}),
	\end{align*}
	and without loss of generality, we may assume $W(t,a,b) = \phi(t,a,b)$. This implies that $W(\bar{t},\bar{a},\bar{b}) \geq \phi  (\bar{t},\bar{a},\bar{b})$ for $(\bar{t},\bar{a},\bar{b}) \in \mathcal{O}$ and $(\bar{t},\bar{a},\bar{b}) \neq (t,a,b)$.
	
	Let us assume that $W$ is not a viscosity supersolution. Then there exists a constant $\theta>0$ such that
	\begin{align*}
			- \partial_t \phi(t,\widehat{a})  + H(t,\widehat{a},(\phi, D \phi, D^2 \phi)(t,\widehat{a})) \leq - \theta < 0.
	\end{align*}
	Recall the definition of $H^\prime$ in (\ref{eq_4_2}) and note that $H^\prime \leq  \sup_{u \in U, \alpha \in \mathbb{R}^{r}, \beta \in G^2} H^\prime = H$. Then for any $(u,\alpha,\beta) \in U \times \mathbb{R}^p \times G^2$, we have
	\begin{align}
	\label{eq_4_4}
	- \partial_t \phi(t,\widehat{a}) + H^\prime (t,\widehat{a},(\phi,D\phi,D^2\phi)(t,\widehat{a}); u,\alpha,\beta) \leq - \theta < 0.
	\end{align}
	
	On the other hand, the DPP in Proposition \ref{Proposition_1} implies
	\begin{align*}
	\phi(t,a,b)  & = W(t,a,b) \\
	& \geq \mathop{\inf_{u \in \mathcal{U}}}_{\alpha \in \mathcal{A},\beta \in \mathcal{B}} \mathbb{E} \Bigl [\int_{t}^{t+\tau} d(x_s^{t,a;u},\Gamma) \dd s  +  \phi(t+\tau,x_{t+\tau}^{t,a;u}, y_{t+\tau;t,a,b}^{u,\alpha,\beta}) \Bigr ],	
	\end{align*}
	and for each $\epsilon > 0
	$, there exist $(u^\epsilon,\alpha^\epsilon,\beta^\epsilon) \in \mathcal{U} \times \mathcal{A} \times \mathcal{B}$ such that
	\begin{align}
	\label{eq_4_5}	
	- \epsilon \tau \leq 	\phi(t,a,b) - \mathbb{E} \Bigl [\int_{t}^{t+\tau} d(x_s^{t,a;u^\epsilon},\Gamma) \dd s  +  \phi(t+\tau,x_{t+\tau}^{t,a;u^\epsilon}, y_{t+\tau;t,a,b}^{u^\epsilon,\alpha^\epsilon,\beta^\epsilon}) \Bigr ]	.
	\end{align}
	
	As in the viscosity subsolution case, we apply It\^o's formula to (\ref{eq_4_5}) and then multiply $\frac{1}{\tau}$. Since (\ref{eq_4_4}) holds for any $(u,\alpha,\beta) \in U \times \mathbb{R}^p \times G^2$, by letting $\tau \downarrow 0$ and noting the arbitrariness of $\epsilon$, we have
	\begin{align*}
		0 & \leq - \partial_t \phi(t,\widehat{a}) + H^\prime (t,\widehat{a},(\phi,D\phi,D^2\phi)(t,\widehat{a}); u,\alpha,\beta) \leq - \theta.
	\end{align*}
	This leads to the desired contradiction, since $\theta > 0$. Hence, $W$ is the viscosity supersolution. This, together with (\ref{eq_4_3}), shows that $W$ is the continuous viscosity solution of (\ref{eq_4_1}). This completes the proof.	
\end{proof}

\section{Uniqueness of Viscosity Solution} \label{Section_5}
We state the comparison principle of viscosity subsolution and supersolution.

\begin{theorem}\label{Theorem_3}
Suppose that Assumptions \ref{Assumption_1}, \ref{Assumption_2} and \ref{Assumption_3} hold. Let $\underline{W} \in C(\bar{\mathcal{O}})$ be the viscosity subsolution of the HJB equation in (\ref{eq_4_1}), and $\overline{W} \in C(\bar{\mathcal{O}})$ the viscosity supersolution of (\ref{eq_4_1}), where both  $\underline{W}$ and $\overline{W}$ satisfy the linear growth condition in $a \in \mathbb{R}^n$. Then
\begin{align}
\label{eq_5_1_11}
\underline{W}(t,\widehat{a}) \leq \overline{W}(t,\widehat{a}),~ \forall (t,\widehat{a}) \in \bar{\mathcal{O}}.
\end{align}	
\end{theorem}

The proof of Theorem \ref{Theorem_3} is reported in Section \ref{Section_6}. Based on Theorem \ref{Theorem_3}, we state the uniqueness of the viscosity solution.
\begin{corollary}\label{Corollary_2}
Assume that Assumptions \ref{Assumption_1}, \ref{Assumption_2} and \ref{Assumption_3} hold. Then $W$ in (\ref{eq_10}) is a unique continuous viscosity solution of the HJB equation in (\ref{eq_4_1}). 
\end{corollary}
\begin{proof}
  	In view of Theorem \ref{Theorem_2}, the value function $W$ in (\ref{eq_10}) is the viscosity solution of the HJB equation in (\ref{eq_4_1}). Note that since $W$ satisfies the linear growth condition from Lemma \ref{Lemma_4} and $W$ is both the viscosity subsolution and the supersolution of (\ref{eq_4_1}), the uniqueness follows from Theorem \ref{Theorem_3}. This completes the proof.
\end{proof}

\subsection*{Concluding Remarks} We have studied the state-constrained stochastic optimal problem for jump diffusion systems. Our main results are Theorems \ref{Theorem_1}, \ref{Theorem_2} and \ref{Theorem_3}, where we have shown that the original value function $V$ in (\ref{eq_4}) can be characterized by the zero-level set of the auxiliary value function $W$ in (\ref{eq_10}) (see (\ref{eq_3_7_5_3_2_1_1_1})). Note that $W$ can be characterized by solving the associated HJB equation in (\ref{eq_4_1}), since $W$ is a unique continuous viscosity solution of (\ref{eq_4_1}).

One possible potential future research problem would be to consider the two-player stochastic game framework, for which we need to generalize Theorem \ref{Theorem_1} using the notion of \emph{nonanticipative strategies}. The state-constrained problem with general BSDE (backward SDE) type recursive objective functionals would also be an interesting avenue to pursue. Applications to various mathematical finance problems will be studied in the near future.

 \section{Proof of Theorem \ref{Theorem_3}}\label{Section_6}
 This section is devoted to the proof of Theorem \ref{Theorem_3}.
 
\subsection{Equivalent Definitions of Viscosity Solutions}
To prove the uniqueness, we first provide two equivalent definitions of Definition \ref{Definition_1}. The HJB equation in (\ref{eq_4_1}) can be rewritten as follows:
\begin{align}
\label{eq_5_1}
\begin{cases}
\sup_{u \in U} \Bigl \{ \sup_{ \alpha \in \mathbb{R}^{r}} H^{(1)}(t,a,(DW,D^2W)(t,\widehat{a}); u,\alpha) \\
\qquad \qquad + \sup_{\beta \in G^2 }H^{(2)}(t,\widehat{a},(W,DW)(t,\widehat{a}); u,\beta) \Bigr \} = 0,~   (t,\widehat{a}) \in \mathcal{O}\\
W(T,\widehat{a}) = \max \{ m(a) - b \},~  (a,b) \in \mathbb{R}^n \times [0,\infty) \\
W(t,a,0) = W_0(t,a),~  (t,a) \in [0,T) \times \mathbb{R}^n,
\end{cases}
\end{align}
where with $D^2 W = \begin{bmatrix}
 	D^2 W_{(11)} & D^2 W_{(12)} \\
 	(D^2 W_{(12)})^\top & D^2W_{(22)} 
 \end{bmatrix}$, 
\begin{align*}
& H^{(1)}(t,a,(\partial_t W, DW,D^2W); u,\alpha) \\
& := - \partial_t W - d(a,\Gamma) - \langle D W, \widehat{f}(t,a,u)
 \rangle  - \frac{1}{2} \Tr (\sigma \sigma^\top (t,a,u) D^2 W_{(11)}) \nonumber \\
&~~~ - \alpha^\top \sigma^\top (t,a,u) D^2 W_{(12)} - \frac{1}{2} |\alpha|^2 D^2 W_{(22)},  \nonumber
\end{align*}
and
 \begin{align*}
   H^{(2)}(t,\widehat{a},(W,DW)(t,\widehat{a}); u,\beta)	:= & - \int_{E} [ W(t,a + \chi(t,a,u,e), b + \beta(e)) - W(t,\widehat{a}) ]\pi( \dd e) 
 \nonumber \\
 & + \int_{E}  \langle D W(t,\widehat{a}),  \widehat{\chi}(t,a,u,e,\beta)
 \rangle \pi(\dd e). 
 \end{align*}
 
 To avoid the possibility of $\sup_{\alpha \in \mathbb{R}^r} H^{(1)}=\infty$ due to the unboundedness of $\alpha$, we have the following result. The proof is analogous that for \cite[Lemma 4.1, Remark 4.5]{Bokanowski_SICON_2016} and \cite[Section 2.3]{Bruder_Hal_2005}.

\begin{lemma}\label{Lemma_9}
$H^{(1)}$ can be expressed as 
\begin{align*}
& \sup_{\alpha \in \mathbb{R}^r} H^{(1)}(t,a,(\partial_t W, DW,D^2W); u,\alpha) = \Lambda^{+}(\mathcal{G}_{\psi}(t,a,(\partial_t W, DW,D^2W);u)),
\end{align*}
where $\Lambda^{+}(A) := \sup_{|v|=1} |A v| = \sup_{v \neq 0} $, i.e., the largest eigenvalue of $A \in \mathbb{S}^n$, and
\begin{align*}
& \mathcal{G}_{\psi}(t,a,(\partial_t W, DW,D^2W);u) := \begin{bmatrix}
\mathcal{G}_{(11)} & \psi(b) \mathcal{G}_{(12)} \\
\psi(b) \mathcal{G}_{(12)}^\top & \psi^2(b) \mathcal{G}_{(12)}
 \end{bmatrix}
\end{align*}
with $\psi:[0,\infty) \rightarrow [0,\infty)$ being a continuous function and
\begin{align*}
\mathcal{G}_{(11)} & := 	- \partial_t W - d(a,\Gamma) - \langle D W, \widehat{f}(t,a,u)
 \rangle  - \frac{1}{2} \Tr (\sigma \sigma^\top (t,a,u) D^2 W_{(11)}) \\
\mathcal{G}_{(12)} & := - \frac{1}{2} (\sigma^\top (t,a,u) D^2 W_{(12)})^\top,~ \mathcal{G}_{(22)} := - \frac{1}{2} D^2 W_{(22)} I_{r}.
\end{align*}
\end{lemma}

\begin{remark}\label{Remark_8}
From Lemma \ref{Lemma_9}, the HJB equation in (\ref{eq_5_1}) is equivalent to
\begin{align}
\label{eq_5_4}
\begin{cases}
 \sup_{u \in U} \Bigl \{	\Lambda^{+}(\mathcal{G}_{\psi}(t,a,(\partial_t W,DW,D^2W)(t,\widehat{a});u)) \\
\qquad \qquad  + \sup_{\beta \in G^2 }H^{(2)}(t,\widehat{a},(W,DW)(t,\widehat{a}); u,\beta) \Bigr \} = 0, ~ (t,\widehat{a}) \in \mathcal{O}\\
W(T,\widehat{a}) = \max \{ m(a) - b \},~  (a,b) \in \mathbb{R}^n \times [0,\infty) \\
W(t,a,0) = W_0(t,a),~  (t,a) \in [0,T) \times \mathbb{R}^n.
\end{cases}
\end{align}
We will use (\ref{eq_5_4}) to prove the comparison principle in Theorem \ref{Theorem_3} with $\psi(b) := \max \{1,b\}$ for $b \in [0,\infty)$.
\end{remark}

For $\delta > 0$, let $E_{\delta} := \{ e \in E~|~ |e| < \delta\}$; hence, $E = E_{\delta} \cup E_{\delta}^C$. We then define
  \begin{align*}
  & H^{(2)}(t,\widehat{a},(W,DW); u,\beta) \\
  & =  H^{(21)}_{\delta}(t,\widehat{a},(W,DW);u,\beta) + H^{(22)}_{\delta}(t,\widehat{a},(W,DW);u,\beta),
  \end{align*}
  where
 \begin{align*}
 H^{(21)}_{\delta}(t,\widehat{a},(W,DW);u,\beta) := & -	\int_{E_{\delta}} [ W(t,a + \chi(t,a,u,e), b + \beta(e))  - W(t,\widehat{a}) ]\pi( \dd e)  \\
 & + \int_{E_{\delta}}  \langle D W(t,\widehat{a}),  \widehat{\chi}(t,a,u,e,\beta)
 \rangle \pi(\dd e), 
 \end{align*}
 and
 \begin{align*}
H^{(22)}_{\delta}(t,\widehat{a},(W,DW);u,\beta)
  := & -	\int_{E_{\delta}^C} [ W(t,a + \chi(t,a,u,e), b + \beta(e)) - W(t,\widehat{a}) ]\pi( \dd e) \\
 &+ \int_{E_{\delta}^C}  \langle D W(t,\widehat{a}),  \widehat{\chi}(t,a,u,e,\beta)
 \rangle \pi(\dd e). 
 \end{align*}

From \cite{Barles_SSR_1997, Barles_Poincare_2008, Buckdahn_SPTA_2011, Peng_Acta_2008, Pham_JMSEC_1998} (see \cite[Proposition 1]{Barles_Poincare_2008}), we have the following first equivalent definition of Definition \ref{Definition_1}:
\begin{lemma}\label{Lemma_7}
Suppose that $W$ is a viscosity subsolution (resp. supersolution) of the HJB equation in (\ref{eq_5_4}). Then it is necessary and sufficient to hold the following: 
\begin{enumerate}[(i)]
\item $W(T,\widehat{a}) \leq \max\{m(a) - b\}$  (resp. $W(T,\widehat{a}) \geq \max\{m(a) - b\}$) for $(a,b) \in \mathbb{R}^n \times [0,\infty)$ and $W(t,a,0) \leq W_0(t,a)$ (resp. $W(t,a,0) \geq W_0(t,a) $) for $(t,a) \in [0,T) \times \mathbb{R}^n$;
\item For all $\delta \in (0,1)$ and test functions $\phi \in C_b^{1,3}(\bar{\mathcal{O}}) \cap C_2(\bar{\mathcal{O}}) $, the following inequality holds at the global maximum (resp. minimum) point $(t,\widehat{a})  \in \mathcal{O}$ of $W-\phi$:
\begin{align*}
& \sup_{u \in U} \Bigl \{ 
\Lambda^{+}(\mathcal{G}_{\psi}(t,a,(\partial_t \phi, D\phi,D^2\phi)(t,\widehat{a});u)) \\
& ~~~ + \sup_{\beta \in G^2 } \bigl \{ H^{(21)}_{\delta}(t,\widehat{a},(\phi,D\phi)(t,\widehat{a});u,\beta)   + H^{(22)}_{\delta}(t,\widehat{a},(W,D\phi)(t,\widehat{a});u,\beta) \bigr \} \Bigr \} \leq  0~ \\
& ( \text{resp. }\\
&  \sup_{u \in U} \Bigl \{ 
\Lambda^{+}(\mathcal{G}_{\psi}(t,a,(\partial_t \phi, D\phi,D^2\phi)(t,\widehat{a});u)) \\
& ~~~ + \sup_{\beta \in G^2 } \bigl \{ H^{(21)}_{\delta}(t,\widehat{a},(\phi,D\phi)(t,\widehat{a});u,\beta)   + H^{(22)}_{\delta}(t,\widehat{a},(W,D\phi)(t,\widehat{a});u,\beta) \bigr \} \Bigr \} \geq 0).
\end{align*}
\end{enumerate}
\end{lemma}

The definition of parabolic superjet and subjet is given as follows \cite{Crandall_AMS_1992_Viscosity}:

\begin{definition}\label{Definition_2}
\begin{enumerate}[(i)]
\item For $W(t,\widehat{a})$, the superjet of $W$ at the point of $(t,\widehat{a}) \in \mathcal{O}$ is defined by
	\begin{align*}
	& \mathcal{P}^{1,2,+} W(t,\widehat{a}) := \{ (q,p,P) \in \mathbb{R} \times	\mathbb{R}^{n+1} \times \mathbb{S}^{n+1} ~|~  \\
	& \qquad \qquad \qquad W(t^\prime,\widehat{a}^\prime) \leq W(t,\widehat{a}) + q(s-t) + \langle p,\widehat{a}^\prime - \widehat{a} \rangle \\
	& \qquad \qquad \qquad + \frac{1}{2} \langle P (\widehat{a}^\prime - \widehat{a}),  \widehat{a}^\prime - \widehat{a} \rangle + o(|t^\prime-t| + |\widehat{a}^\prime - \widehat{a}|^2) \}.
	\end{align*}
	\item The closure of $\mathcal{P}^{1,2,+} W(t,\widehat{a})$ is defined by
	\begin{align*}
		& \overline{\mathcal{P}}^{1,2,+}W(t,\widehat{a}) := \{(q,p,P) \in \mathbb{R} \times	\mathbb{R}^{n+1} \times \mathbb{S}^{n+1} ~|~  \\
		& \qquad \qquad \qquad   (q,p,P) = \lim_{n \rightarrow \infty} (q_n,p_n,P_n)~\text{with}~ (q_n,p_n,P_n) \in \mathcal{P}^{1,2,+} W(t_n,\widehat{a}_n) \\
		& \qquad \qquad \qquad \text{and}~ \lim_{n \rightarrow \infty} (t_n,\widehat{a}_n,W(t_n,\widehat{a}_n)) = (t,\widehat{a},W(t,\widehat{a}))\}.
	\end{align*}
	\item For $W(t,\widehat{a})$, the subjet of $W$ at the point of $(t,\widehat{a}) \in \mathcal{O}$ and its closure are defined by
	\begin{align*}
	\mathcal{P}^{1,2,-} W(t,\widehat{a}) := - \mathcal{P}^{1,2,+} (-W(t,\widehat{a})),~\overline{\mathcal{P}}^{1,2,-} W(t,\widehat{a}) := - \overline{\mathcal{P}}^{1,2,+} (-W(t,\widehat{a})). 
	\end{align*}
	\end{enumerate}
\end{definition}

Using Definition \ref{Definition_2} and Lemma \ref{Lemma_7}, we have the following second equivalent definition of Definition \ref{Definition_1} (see \cite{Barles_SSR_1997, Pham_JMSEC_1998}, \cite[Lemma 3.5]{Peng_Acta_2008}, \cite[Proposition 1]{Barles_Poincare_2008}, and \cite[Lemmas 5.4 and 5.5, Chapter 4]{Yong_book}):

\begin{lemma}\label{Lemma_8}
	Suppose that $W$ is a viscosity subsolution (resp. supersolution) of the HJB equation in (\ref{eq_5_4}). Then it is necessary and sufficient to hold the following:
\begin{enumerate}[(i)]
\item $W(T,\widehat{a}) \leq \max\{m(a) - b\}$  (resp. $W(T,\widehat{a}) \geq \max\{m(a) - b\}$) for $(a,b) \in \mathbb{R}^n \times [0,\infty)$ and $W(t,a,0) \leq W_0(t,a)$ (resp. $W(t,a,0) \geq W_0(t,a) $) for $(t,a) \in [0,T) \times \mathbb{R}^n$;
\item For all $\delta \in (0,1)$ and test functions $\phi \in C_b^{1,3}(\bar{\mathcal{O}}) \cap C_2(\bar{\mathcal{O}})$ with the local maximum (resp. minimum) point $(t,\widehat{a}) \in \mathcal{O}$ of $W-\phi$, if $(q,p,P) \in \overline{\mathcal{P}}^{1,2,+}W(t,\widehat{a})$ (resp. $(q,p,P) \in \overline{\mathcal{P}}^{1,2,-}W(t,\widehat{a})$) with $p = D \phi(t,\widehat{a})$ and $P = D^2 \phi(t,\widehat{a})$,
then the following inequality holds:
	\begin{align*}
	&\sup_{u \in U} \Bigl \{ 
	\Lambda^{+}(\mathcal{G}_{\psi}(t,a,(q,p,P);u)) \\
	&\qquad  + \sup_{\beta \in G^2 } \bigl \{ H^{(21)}_{\delta}(t,\widehat{a},(\phi,D \phi)(t,\widehat{a});u,\beta)  + H^{(22)}_{\delta}(t,\widehat{a},W(t,\widehat{a}),p;u,\beta) \bigr \} \Bigr \} \leq 0\\
	& ( \text{resp.} \\
	&  \sup_{u \in U} \Bigl \{ 
		\Lambda^{+}(\mathcal{G}_{\psi}(t,a,(q,p,P);u))\\
	&\qquad  + \sup_{\beta \in G^2 } \bigl \{ H^{(21)}_{\delta}(t,\widehat{a},(\phi,D \phi)(t,\widehat{a});u,\beta)  + H^{(22)}_{\delta}(t,\widehat{a},W(t,\widehat{a}),p;u,\beta) \bigr \} \Bigr \} \geq 0).
	\end{align*}
	\end{enumerate}	
\end{lemma}

\begin{remark}\label{Remark_9}
Lemma \ref{Lemma_8} is introduced due to the singularity of the L\'evy measure in zero, appearing in the nonlocal operator $H_{\delta}^{(21)}$. We will see that with the regularity of the test function, one can pass the limit of $H_{\delta}^{(21)}$ around the singular point of the measure.
\end{remark}

\subsection{Strict Viscosity Subsolution}
\begin{lemma}\label{Lemma_10}
Suppose that $\underline{W}(t,\widehat{a})$ is the viscosity subsolution of (\ref{eq_5_4}). Let
\begin{align*}	
\underline{W}_{\nu}(t,\widehat{a}) := \underline{W}(t,\widehat{a}) + \nu \gamma(t,b),
\end{align*}
where for $\nu >0$,
\begin{align*}
\gamma(t,b) := -(T-t) - \log (1+b).
\end{align*}
Then $\underline{W}_{\nu}$ is the strict viscosity subsolution of (\ref{eq_5_4}) in the sense that $\leq 0$ is replaced by $\leq - \frac{\nu}{8}$ in Definition \ref{Definition_1}.
\end{lemma}
\begin{proof}
	We first verify the boundary condition of $W_{\nu}$. Note that
	\begin{align*}
	\underline{W}_{\nu}(T,\widehat{a}) = 	\underline{W}(T,\widehat{a}) - \nu \log(1+b)
	\leq \max\{m(a) - b\},
	\end{align*}
	and by Lemma \ref{Lemma_6}
	\begin{align*}
	\underline{W}_{\nu}(t,a,0) &= \underline{W}(t,a,0) - \nu(T-t) \leq W_0(t,a).
	\end{align*}
	
	Now, let $\phi_{\nu} \in C_b^{1,3}(\bar{\mathcal{O}})$ be the test function such that
	\begin{align*}
	(\underline{W}_{\nu}- \phi_{\nu})(t,\widehat{a}) = \max_{(t^\prime,\widehat{a}^\prime) \in \mathcal{O}}(\underline{W}_{\nu}- \phi_{\nu})(t^\prime,\widehat{a}^\prime).
	\end{align*}
	Then from (\ref{eq_5_4}) and Definition \ref{Definition_1}, it is necessary to show that
\begin{align}
\label{eq_5_7_1}
& \sup_{u \in U} \Bigl \{	\Lambda^{+}(\mathcal{G}_{\psi}(t,a,(\partial_t \phi_{\nu}, D \phi_{\nu}, D^2 \phi_{\nu})(t,\widehat{a});u)) \\
& \qquad \qquad  + \sup_{\beta \in G^2 }H^{(2)}(t,\widehat{a},(\phi_{\nu},D\phi_{\nu})(t,\widehat{a}); u,\beta) \Bigr \} \leq - \frac{\nu}{8}. \nonumber
\end{align}

By defining
\begin{align*}
\underline{\phi}(t,a,b) := 	- \nu \gamma(t,b) + \phi_{\nu}(t,a,b),
\end{align*}
it is easy to see that $\phi \in C_b^{1,3}(\bar{\mathcal{O}})$ and 
\begin{align*}
(\underline{W}_{\nu}- \phi_{\nu})(t,\widehat{a}) & = \underline{W}(t,\widehat{a}) - (- \nu \gamma(t,b) + \phi_{\nu}(t,\widehat{a})) \\
	& = \underline{W}(t,\widehat{a}) - \underline{\phi}(t,\widehat{a}).	
\end{align*}
Then
\begin{align}
\label{eq_5_5}
\max_{(t^\prime,\widehat{a}^\prime) \in \mathcal{O}}(\underline{W}_{\nu}- \phi_{\nu})(t^\prime,\widehat{a}^\prime) & = (\underline{W}_{\nu}- \phi_{\nu})(t,\widehat{a}) \\
	& = (\underline{W} - \underline{\phi})(t,\widehat{a})  = \max_{(t^\prime,\widehat{a}^\prime) \in \mathcal{O}}(\underline{W} - \underline{\phi})(t^\prime,\widehat{a}^\prime). \nonumber
\end{align}

Since $\phi_{\nu} = \underline{\phi} + \nu \gamma$, $\Lambda^+$ is the norm, and $H^{(2)}$ is linear in $\phi_{\nu}$ and $D \phi_{\nu}$, 
\begin{align*}
& \sup_{u \in U} \Bigl \{	\Lambda^{+}(\mathcal{G}_{\psi}(t,a,(\partial_t \phi_{\nu}, D \phi_{\nu}, D^2 \phi_{\nu})(t,\widehat{a});u)) \\
& \qquad \qquad  + \sup_{\beta \in G^2 }H^{(2)}(t,\widehat{a},(\phi_{\nu},D\phi_{\nu})(t,\widehat{a}); u,\beta) \Bigr \} \leq I_{(1)} + I_{(2)}, \nonumber
\end{align*}
where
\begin{align*}
I^{(1)} & := \sup_{u \in U} \Bigl \{ \Lambda^+(\mathcal{G}_{\psi}(t,a,(\partial_t \underline{\phi} , D \underline{\phi} , D^2 \underline{\phi} ) (t,\widehat{a});u))  \\
&\qquad \qquad + \sup_{\beta \in G^2 }H^{(2)}(t,\widehat{a},(\underline{\phi},D \underline{\phi})(t,\widehat{a}); u,\beta) \Bigr \} \\
I^{(2)} & := \nu \sup_{u \in U} \Bigl \{ \Lambda^+(\mathcal{G}_{\psi}(t,a,(\partial_t \gamma , D \gamma , D^2 \gamma ) (t,b);u)) \\
&\qquad \qquad + \sup_{\beta \in G^2 }H^{(2)}(t,\widehat{a},(\gamma ,D \gamma )(t,b); u,\beta) \Bigr \}.
\end{align*}

We now provide the estimate of $ I^{(1)}$ and $ I^{(2)}$. First, since $\underline{W}$ is the viscosity subsolution and $\underline{\phi}$ is the corresponding test function in view of (\ref{eq_5_5}), we have
\begin{align}
\label{eq_5_7}
 I^{(1)} \leq 0.
\end{align}
For $ I^{(2)}$, we observe that
\begin{align*}
& H^{(2)}(t,\widehat{a},(\gamma ,D \gamma )(t,b); u,\beta) \\
 & = \int_{E} [ \log(1+b+\beta(e)) - \log(1+b)] \pi( \dd e) - \int_{E} \frac{1}{1+b} \beta(e)  \pi (\dd e).
\end{align*}
Since $b \in [0,\infty)$, it is easy to see that with $\beta(e) = 0$,  
\begin{align*}
\sup_{\beta \in G^2} H^{(2)}(t,\widehat{a},(\gamma ,D \gamma )(t,b); u,\beta) = 0.
\end{align*}
Recall $\psi(b) = \max \{1,b\}$ for $b \in [0,\infty)$. In the definition of $\mathcal{G}_{\psi}$,
\begin{align*}
\mathcal{G}_{(11)} &= -1 - \frac{1}{1+b} l(t,a,u) - d(a,\Gamma)	\\
\psi(b)\mathcal{G}_{(12)} &= 0,~ \psi^2(b)\mathcal{G}_{(22)} = -\max\{1,b^2\} \frac{1}{2(1+b)^2} I_r .
\end{align*}

Note that $l$ and $d(a,\Gamma)$ are positive for $b \in [0,\infty)$; therefore, $\mathcal{G}_{(11)} \leq -1$. Moreover, for $b \in [0,\infty)$, we can show that $-\frac{1}{2} I_r \leq \mathcal{G}_{(22)} \leq -\frac{1}{8} I_r$. Therefore,
\begin{align*}
& \Lambda^+(\mathcal{G}_{\psi}(t,a,(\partial_t \gamma , D \gamma , D^2 \gamma ) (t,b);u))\\
& = \Lambda^+ \Bigl (\begin{bmatrix}
	-1 - \frac{1}{1+b} l(t,a,u) - d(a,\Gamma) & 0 \\
	0 & -\max\{1,b^2\} \frac{1}{2(1+b)^2} I_r
\end{bmatrix} \Bigr ) \leq - \frac{1}{8},
\end{align*}
which implies
\begin{align}
\label{eq_5_8}
I_{(2)} \leq -(\nu/8). 
\end{align}
Then (\ref{eq_5_7}) and (\ref{eq_5_8}) lead to (\ref{eq_5_7_1}). We complete the proof.
\end{proof}

\subsection{Proof of Theorem \ref{Theorem_3}}

We continue to prove the uniqueness. For $\eta > 0$ and $\nu > 0$, let
\begin{align}
\label{eq_5_9}
\Psi_{\nu;\eta,\lambda} (t,a,b) &:= \underline{W}_{\nu}(t,a,b) - \overline{W}(t,a,b) - 2 \eta e^{-\lambda t}(1 + |a|^2 + b),
\end{align}
where $\lambda > 0$ will be specified later. Then it is necessary to show that
\begin{align}
\label{eq_5_9_1_1_1}
\Psi_{\nu;\eta,\lambda} (t,\widehat{a}) \leq 0,~ \forall (t,\widehat{a}) \in \bar{\mathcal{O}},
\end{align}
since by letting $\eta \downarrow 0$ and then $\nu \downarrow 0$, the desired result in (\ref{eq_5_1_11}) holds, i.e.,
\begin{align*}
\underline{W}(t,\widehat{a}) \leq \overline{W}(t,\widehat{a}),~ \forall \in \bar{\mathcal{O}}.
\end{align*}

Assume that (\ref{eq_5_9_1_1_1}) is not true, i.e., $\Psi_{\nu;\eta,\lambda} (t,\widehat{a}) > 0$ for some $(t,\widehat{a}) \in \bar{\mathcal{O}} > 0$. Consider,
\begin{align}
\label{eq_5_10}
\Psi_{\nu;\eta,\lambda}(\tilde{t},\tilde{a},\tilde{b}) = \max_{(t,a,b) \in \bar{\mathcal{O}}} \Psi_{\nu;\eta,\lambda}(t,a,b) > 0,	
\end{align}
where the maximum exists, since $\underline{W}_{\nu}$ and $\overline{W}$ satisfy the linear growth condition ($\log(1+b)$ also holds the linear growth condition) and $e^{-\lambda t}$ is decreasing. Actually, $(\tilde{t},\tilde{a},\tilde{b})$ is dependent on $(\nu,\eta,\lambda)$, i.e., $(\tilde{t},\tilde{a},\tilde{b}) := (\tilde{t}_{\nu;\eta,\lambda},\tilde{a}_{\nu;\eta,\lambda},\tilde{b}_{\nu;\eta,\lambda})$. 

Suppose that $\tilde{t} = T$. Then in view of (\ref{eq_5_9}) and the definition of $\underline{W}_{\nu}$, 
\begin{align*}
\Psi_{\nu;\eta,\lambda}(T,\tilde{a},\tilde{b}) = & \underline{W}(T,\tilde{a},\tilde{b}) + \nu \gamma(T,\tilde{b}) - \overline{W}(T,\tilde{a},\tilde{b}) \\	
	& - 2 \eta e^{-\lambda T}(1+|\tilde{a}|^2 + \tilde{b}) \leq 0,
\end{align*}
which contradicts (\ref{eq_5_10}). Hence, $\tilde{t} < T$. Similarly, when $\tilde{b}=0$, we have
\begin{align*}
\Psi_{\nu;\eta,\lambda}(\tilde{t},\tilde{a},0) & = \underline{W}(\tilde{t},\tilde{a},0) - \nu (T-t) \\
& - \underline{W}(\tilde{t},\tilde{a},0) - 2 \eta e^{-\lambda t}(1+|\tilde{a}|^2) \leq 0,
\end{align*}
which again contradicts (\ref{eq_5_10}). Hence, $\tilde{b} > 0$. This implies that $(\tilde{t},\tilde{a},\tilde{b}) \in \mathcal{O}$.

After doubling variables of $\Psi$, we consider 
\begin{align*}
\Psi_{\nu;\eta,\lambda}^{\kappa} (t,a,b,\breve{a},\breve{b}) 
= & \widehat{\Psi}_{\nu;\eta,\lambda}(t,a,b,\breve{a},\breve{b}) - \kappa \zeta(a,b,\breve{a},\breve{b}),
\end{align*}
where $\kappa >0$ and
\begin{align*}
\widehat{\Psi}_{\nu;\eta,\lambda}(t,a,b,\breve{a},\breve{b}) := & \underline{W}_{\nu}(t,a,b) - \overline{W}(t,\breve{a},\breve{b}) - \eta e^{-\lambda t}(1+|a|^2  + b) \\ 
& - \eta e^{-\lambda t}(1+|\breve{a}|^2  + \breve{b}) \\
& - \frac{\eta e^{-\lambda t} }{2} \Bigl (|a - \tilde{a}|^2 + (b-\tilde{b}) \Bigr )  - \frac{1}{2} |t-\tilde{t}|^2 \\
\zeta(a,b,\breve{a},\breve{b}) := & \frac{1}{2} \Bigl ( |a - \breve{a}|^2 + |b-\breve{b}|^2 \Bigr ).
\end{align*}
Since $\widehat{\Psi}_{\nu;\eta,\lambda}(t,a,b,a,b) \leq \Psi_{\nu;\eta,\lambda}(t,a,b)$ and $\widehat{\Psi}_{\nu;\eta,\lambda}(\tilde{t},\tilde{a},\tilde{b},\tilde{a},\tilde{b}) = \Psi_{\nu;\eta,\lambda}(\tilde{t},\tilde{a},\tilde{b})$, 
\begin{align}
\label{eq_5_11}
\Psi_{\nu;\eta,\lambda}(\tilde{t},\tilde{a},\tilde{b}) = \max_{(t,a,b) \in \mathcal{O}} \Psi_{\nu;\eta,\lambda}(t,a,b) = 	\max_{(t,a,b) \in \mathcal{O}} \widehat{\Psi}_{\nu;\eta,\lambda}(t,a,b,a,b).
\end{align}

We consider $(t^\prime_{\kappa},a^\prime_{\kappa}, b^\prime_{\kappa},  \breve{a}^\prime_{\kappa}, \breve{b}^\prime_{\kappa} )$ such that
\begin{align*}
\Psi_{\nu;\eta,\lambda}^{\kappa}(t^\prime_{\kappa},a^\prime_{\kappa}, b^\prime_{\kappa},  \breve{a}^\prime_{\kappa}, \breve{b}^\prime_{\kappa})  = \max_{(t,a, b,\breve{a}, \breve{b}) \in  \mathcal{O}  \times \mathbb{R}^n \times (0,\infty)} 
\Bigl \{ \widehat{\Psi}_{\nu;\eta,\lambda}(t,a,b,\breve{a},\breve{b}) - \kappa \zeta(a,b,\breve{a},\breve{b}) \Bigr \},
\end{align*}
which exists since $-\Psi_{\nu;\eta,\lambda}^{\kappa}$ is coercive. Then from \cite[Proposition 3.7]{Crandall_AMS_1992_Viscosity}, 
\begin{align*}
\begin{cases}
	\lim_{\kappa \rightarrow \infty} \kappa \zeta (a^\prime_{\kappa}, b^\prime_{\kappa},  \breve{a}^\prime_{\kappa}, \breve{b}^\prime_{\kappa}) = 0  \\
	\lim_{\kappa \rightarrow \infty} \Psi_{\nu;\eta,\lambda}^{\kappa}(t^\prime_{\kappa}, a^\prime_{\kappa}, b^\prime_{\kappa},  \breve{a}^\prime_{\kappa}, \breve{b}^\prime_{\kappa}) = \widehat{\Psi}_{\nu;\eta,\lambda}(t^\prime,a^\prime, b^\prime,  \breve{a}^\prime , \breve{b}^\prime)  \\
	\qquad \qquad \qquad \qquad \qquad \qquad \qquad = \max_{\zeta(a,b,\breve{a},\breve{b}) = 0} \widehat{\Psi}_{\nu;\eta,\lambda}(t,a, b,  \breve{a}, \breve{b}) \\
\lim_{\kappa \rightarrow \infty}  \zeta (a^\prime_{\kappa}, b^\prime_{\kappa},  \breve{a}^\prime_{\kappa}, \breve{b}^\prime_{\kappa}) = \zeta (a^\prime,b^\prime,\breve{a}^\prime,\breve{b}^\prime) = 0.
\end{cases}
\end{align*}
This, together with (\ref{eq_5_11}), implies that as $\kappa \rightarrow \infty$,
\begin{align}
\label{eq_5_14}
\begin{cases}
|a^\prime_{\kappa} - \breve{a}^\prime_{\kappa}|^2,~|b^\prime_{\kappa} - \breve{b}^\prime_{\kappa}|^2 \rightarrow 0 \\ 
\frac{\kappa}{2} |a^\prime_{\kappa} - \breve{a}^\prime_{\kappa}|^2,~ \frac{\kappa}{2}|b^\prime_{\kappa} - \breve{b}^\prime_{\kappa}|^2 \rightarrow 0 \\
t^\prime_{\kappa} \rightarrow \tilde{t}, ~ a^\prime_{\kappa},\breve{a}^\prime_{\kappa} \rightarrow \tilde{a},~ b^\prime_{\kappa},\breve{b}^\prime_{\kappa} \rightarrow \tilde{b}.
\end{cases}
\end{align}
For simplicity, we denote $(t^\prime,a^\prime, b^\prime,  \breve{a}^\prime , \breve{b}^\prime ) := (t^\prime_{\kappa},a^\prime_{\kappa}, b^\prime_{\kappa},  \breve{a}^\prime_{\kappa}, \breve{b}^\prime_{\kappa} )$. 

We let
\begin{align*}
	h_{\eta,\lambda}(t,a,b) &:= \eta e^{-\lambda t}(1+|a|^2 + b) + \frac{1}{2}|t-\tilde{t}|^2  + \eta e^{-\lambda t} \frac{1}{2} \Bigl (|a - \tilde{a}|^2 + (b-\tilde{b}) \Bigr ) \\
	\widehat{h}_{\eta,\lambda}(t,\breve{a},\breve{b}) & := \eta e^{-\lambda t}(1+|\breve{a}|^2 +\breve{b}) \\
	\zeta_{\kappa}(a,b,\breve{a},\breve{b}) & := \frac{\kappa}{2} \Bigl ( |a - \breve{a}|^2 + |b-\breve{b}|^2 \Bigr ).
\end{align*}
Then
\begin{align}
\label{eq_5_15_111}
\Psi_{\nu;\eta,\lambda}^{\kappa} 	(t,a,\breve{a},b,\breve{b}) =& (\underline{W}_{\nu}(t,a,b) - h_{\eta,\lambda}(t,a,b)) \\
& - (\overline{W}(t,\breve{a},\breve{b}) + \widehat{h}(t,\breve{a},\breve{b})) - \zeta_{\kappa}(a,b,\breve{a},\breve{b}). \nonumber
\end{align}
We invoke Crandall-Ishii's lemma in \cite[Theorem 8.3 and Remark 2.7]{Crandall_AMS_1992_Viscosity}, from which there exist
\begin{align*}
\begin{cases}
q+\widehat{q} = \partial_t \zeta(a^\prime, b^\prime,  \breve{a}^\prime , \breve{b}^\prime ) = 0	\\
(q + \partial_t h_{\eta,\lambda}, D_{(a,b)}(h_{\eta,\lambda} + \zeta_{\kappa}), P + D_{(a,b)}^2 h_{\eta,\lambda})(t^\prime, a^\prime, b^\prime) \in \overline{\mathcal{P}}^{1,2,+}\underline{W}_{\nu}(t^\prime, a^\prime, b^\prime) \\
(- \widehat{q} - \partial_t \widehat{h}_{\eta,\lambda}, - D_{(\breve{a},\breve{b})}( \widehat{h}_{\eta,\lambda} + \zeta_{\kappa}), -\widehat{P} - D_{(\breve{a},\breve{b})}^2 \widehat{h}_{\eta,\lambda})(t^\prime, \breve{a}^\prime, \breve{b}^\prime) \in \overline{\mathcal{P}}^{1,2,-}\overline{W}(t^\prime, \breve{a}^\prime, \breve{b}^\prime),
\end{cases}
\end{align*}
such that
\begin{align}
\label{eq_5_16_1_1_1_1}
-3 \kappa \begin{bmatrix}
 	I_{n+1} & 0 \\
 	0 & I_{n+1}
 \end{bmatrix} \leq \begin{bmatrix}
 	P & 0 \\
 	0 & \widehat{P}
 \end{bmatrix}	\leq 3 \kappa \begin{bmatrix}
 	I_{n+1} & -I_{n+1} \\
 	-I_{n+1} & I_{n+1}
 \end{bmatrix}.
\end{align}

Straightforward computation yields
\begin{align}
\label{eq_5_15_1_1_1}
\begin{cases}
\partial_t h_{\eta,\lambda}(t,a,b) = -\eta \lambda e^{- \lambda t}(1+|a|^2+b) + (t-\tilde{t}) -  \frac{\eta\lambda e^{-\lambda t}}{2}\Bigl (|a - \tilde{a}|^2 + (b-\tilde{b}) \Bigr ) \\
\partial_t \widehat{h}_{\eta,\lambda}(t,\breve{a},\breve{b}) = -\eta \lambda  e^{-\lambda t}(1 + |\breve{a}|^2 +\breve{b}) \\
D_{(a,b)} h_{\eta,\lambda}(t,a,b) = \begin{bmatrix}
2 \eta e^{-\lambda t} a + \eta e^{-\lambda t} (a - \tilde{a}) 	\\
\frac{3}{2} \eta e^{-\lambda t}
 \end{bmatrix},~ D_{(\breve{a},\breve{b})} \widehat{h}_{\eta,\lambda}(t,\breve{a},\breve{b}) = \begin{bmatrix}
2 \eta e^{-\lambda t}   \breve{a} \\
\eta e^{-\lambda t}
 \end{bmatrix} \\
 D_{(a,b)}^2 h_{\eta,\lambda}(t,a,b) = \begin{bmatrix}
3 \eta e^{-\lambda t} I_n & 0  	\\
0 & 0
 \end{bmatrix},~ D_{(\breve{a},\breve{b})}^2 \widehat{h}_{\eta,\lambda}(t,\breve{a},\breve{b}) = \begin{bmatrix}
2 \eta e^{-\lambda t}   I_n & 0 \\
0 & 0
 \end{bmatrix} \\
 D_{(a,b)} \zeta_{\kappa}(t,a,b,\breve{a},\breve{b}) = \begin{bmatrix}
 	\kappa (a - \breve{a}) \\
 	\kappa (b - \breve{b})
 \end{bmatrix},~ D_{(\breve{a},\breve{b})} \zeta_{\kappa}(t,a,b,\breve{a},\breve{b}) = \begin{bmatrix}
 	- \kappa (a - \breve{a}) \\
 	- \kappa (b - \breve{b})
 \end{bmatrix}.
 \end{cases}
\end{align}
Below, we use the superscript $^\prime$ in the above derivatives when they are evaluated at $(t^\prime,a^\prime, b^\prime,  \breve{a}^\prime , \breve{b}^\prime )$ (e.g. $\partial_t h_{\eta,\lambda}^\prime := \partial_t h_{\eta,\lambda}(t^\prime,a^\prime,b^\prime)$).

From Lemmas \ref{Lemma_8} and \ref{Lemma_10}, there exists $\phi \in C_b^{1,3}(\bar{\mathcal{O}}) \cap C_2(\bar{\mathcal{O}})$ such that
\begin{align*}
	&\sup_{u \in U} \Bigl \{ \Lambda^{+}(\mathcal{G}_{\psi}(t^\prime,a^\prime,(q + \partial_t h_{\eta,\lambda}^\prime, D_{(a,b)}(h_{\eta,\lambda}^\prime + \zeta^\prime_{\kappa}), P + D_{(a,b)}^2 h_{\eta,\lambda}^\prime); u)) \\
	&\qquad \qquad + \sup_{\beta \in G^2 } \bigl \{ H^{(21)}_{\delta}(t^\prime, a^\prime, b^\prime,(\phi,D \phi)(t^\prime, a^\prime, b^\prime);u,\beta) \\
	&\qquad \qquad + H^{(22)}_{\delta}(t^\prime, a^\prime, b^\prime, \underline{W}_{\nu}(t,a^\prime, b^\prime), D_{(a,b)}(h_{\eta,\lambda}^\prime + \zeta^\prime_{\kappa}) ;u,\beta) \bigr \} \Bigr \} \leq - \frac{\nu}{8},
\end{align*}
and
\begin{align*}
	&\sup_{u \in U} \Bigl \{ \Lambda^{+}(\mathcal{G}_{\psi}(t^\prime, \breve{a}^\prime,(- \widehat{q} - \partial_t \widehat{h}_{\eta,\lambda}^\prime, - D_{(\breve{a},\breve{b})}( \widehat{h}_{\eta,\lambda}^\prime + \zeta_{\kappa}^\prime), -\widehat{P} - D_{(\breve{a},\breve{b})}^2 \widehat{h}_{\eta,\lambda}^\prime); u)) \\
	&\qquad \qquad + \sup_{\beta \in G^2 } \bigl \{ H^{(21)}_{\delta}(t^\prime, \breve{a}^\prime, \breve{b}^\prime,(\phi,D \phi)(t^\prime, \breve{a}^\prime, \breve{b}^\prime);u,\beta) \\
	&\qquad \qquad + H^{(22)}_{\delta}(t^\prime, \breve{a}^\prime, \breve{b}^\prime, \overline{W}(t^\prime, \breve{a}^\prime, \breve{b}^\prime),- D_{(\breve{a},\breve{b})}( \widehat{h}_{\eta,\lambda}^\prime + \zeta_{\kappa}^\prime);u,\beta) \bigr \} \Bigr \} \geq 0.
\end{align*}
Then using $\sup\{ f(x) - g(s)\} \leq \sup f(x) - \sup g(x)$, we have
\begin{align*}
\Upsilon^{(1)} + \Upsilon^{(2)} + \Upsilon^{(3)} \geq \frac{\nu}{8},
\end{align*}
where
\begin{align*}	
\Upsilon^{(1)} & := \sup_{u \in U} \Bigl \{ \Lambda^{+}(\mathcal{G}_{\psi}(t^\prime, \breve{a}^\prime,(- \widehat{q} - \partial_t \widehat{h}_{\eta,\lambda}^\prime, - D_{(\breve{a},\breve{b})}( \widehat{h}_{\eta,\lambda}^\prime + \zeta_{\kappa}^\prime), -\widehat{P} - D_{(\breve{a},\breve{b})}^2 \widehat{h}_{\eta,\lambda}^\prime); u)) \\
& \qquad \qquad -  \Lambda^{+}(\mathcal{G}_{\psi}(t^\prime,a^\prime,(q + \partial_t h_{\eta,\lambda}^\prime, D_{(a,b)}(h_{\eta,\lambda}^\prime + \zeta^\prime_{\kappa}), P + D_{(a,b)}^2 h_{\eta,\lambda}^\prime); u))  \Bigr \} \\
\Upsilon^{(2)} & := \sup_{u \in U, \beta \in G^2 } \Bigl \{ H^{(21)}_{\delta}(t^\prime, \breve{a}^\prime, \breve{b}^\prime,(\phi,D \phi)(t^\prime, \breve{a}^\prime, \breve{b}^\prime);u,\beta) \\
&\qquad \qquad -  H^{(21)}_{\delta}(t^\prime, a^\prime, b^\prime,(\phi,D \phi)(t^\prime, a^\prime, b^\prime);u,\beta) \Bigr \} \\
\Upsilon^{(3)} & := \sup_{u \in U, \beta \in G^2 }  \Bigl \{ H^{(22)}_{\delta}(t^\prime, \breve{a}^\prime, \breve{b}^\prime, \overline{W}(t^\prime, \breve{a}^\prime, \breve{b}^\prime),- D_{(\breve{a},\breve{b})}( \widehat{h}_{\eta,\lambda}^\prime + \zeta_{\kappa}^\prime);u,\beta)   \\
&\qquad \qquad - H^{(22)}_{\delta}(t^\prime, a^\prime, b^\prime, \underline{W}_{\nu}(t,a^\prime, b^\prime), D_{(a,b)}(h_{\eta,\lambda}^\prime + \zeta^\prime_{\kappa}) ;u,\beta) \Bigr \}.
\end{align*}

We obtain the estimate of $\Upsilon^{(1)}$, $\Upsilon^{(2)}$ and $\Upsilon^{(3)}$ in (\ref{eq_5_18_2_1_0_2_1_1}), (\ref{eq_5_25_3_2_6_2_1_2}) and (\ref{eq_5_30_2_1_2_1}) separately below. In particular, (\ref{eq_5_18_2_1_0_2_1_1}), (\ref{eq_5_25_3_2_6_2_1_2}) and (\ref{eq_5_30_2_1_2_1}) show that for any $\lambda \geq \max \{C_2, C_4\}$,  where $C_2$ and $C_4$ are given below, we have
\begin{align*}
\frac{\nu}{8} \leq \lim_{\eta \downarrow 0}  \lim_{\kappa \rightarrow \infty} \lim_{\delta \downarrow 0}  \{ \Upsilon^{(1)} + \Upsilon^{(2)} + \Upsilon^{(3)} \} \leq 0,
\end{align*}
which leads to the desired contradiction, since $\nu > 0$ from Lemma \ref{Lemma_10}. Hence, (\ref{eq_5_9_1_1_1}) holds, and we have the comparison principle in (\ref{eq_5_1_11}). 

\subsection{Estimate of $\Upsilon^{(1)}$}
From the definition of $\mathcal{G}_{\psi}$, we denote
\begin{align*}
& \mathcal{G}_{\psi}(t^\prime, \breve{a}^\prime,(- \widehat{q} - \partial_t \widehat{h}_{\eta,\lambda}^\prime, - D_{(\breve{a},\breve{b})}( \widehat{h}_{\eta,\lambda}^\prime + \zeta_{\kappa}^\prime), -\widehat{P} - D_{(\breve{a},\breve{b})}^2 \widehat{h}_{\eta,\lambda}^\prime);u) \\
&  = \widehat{\mathcal{G}}^{(1)}_{\psi} + \widehat{\mathcal{G}}^{(2)}_{\psi} + \widehat{\mathcal{G}}^{(3)}_{\psi},
\end{align*}
where
\begin{align*}
\widehat{\mathcal{G}}^{(1)}_{\psi} & := \mathcal{G}_{\psi}(t^\prime, \breve{a}^\prime,(- \widehat{q} - \frac{1}{2}\partial_t \widehat{h}_{\eta,\lambda}^\prime, - D_{(\breve{a},\breve{b})}( \widehat{h}_{\eta,\lambda}^\prime + \zeta_{\kappa}^\prime),0);u) \\
\widehat{\mathcal{G}}^{(2)}_{\psi} &:= \mathcal{G}_{\psi}(t^\prime, \breve{a}^\prime,(0,0,-\widehat{P});u) \\
\widehat{\mathcal{G}}^{(3)}_{\psi} &:=\mathcal{G}_{\psi}(t^\prime, \breve{a}^\prime,(- \frac{1}{2}\partial_t \widehat{h}_{\eta,\lambda}^\prime,0,- D_{(\breve{a},\breve{b})}^2 \widehat{h}_{\eta,\lambda}^\prime);u),
\end{align*}
and
\begin{align*}
& \mathcal{G}_{\psi}(t^\prime,a^\prime,(q + \partial_t h_{\eta,\lambda}^\prime, D_{(a,b)}(h_{\eta,\lambda}^\prime + \zeta^\prime_{\kappa}), P + D_{(a,b)}^2 h_{\eta,\lambda}^\prime); u) \\
&  = \mathcal{G}^{(1)}_{\psi} + \mathcal{G}^{(2)}_{\psi} + \mathcal{G}^{(3)}_{\psi},
\end{align*}
where
\begin{align*}
\mathcal{G}^{(1)}_{\psi} & := \mathcal{G}_{\psi}(t^\prime,a^\prime,(q + \frac{1}{2} \partial_t h_{\eta,\lambda}^\prime, D_{(a,b)}(h_{\eta,\lambda}^\prime + \zeta^\prime_{\kappa}),0);u) \\
\mathcal{G}^{(2)}_{\psi} & := 	\mathcal{G}_{\psi}(t^\prime,a^\prime,(0,0, P); u) \\
\mathcal{G}^{(3)}_{\psi} & := \mathcal{G}_{\psi}(t^\prime,a^\prime,(\frac{1}{2} \partial_t h_{\eta,\lambda}^\prime, 0, D_{(a,b)}^2 h_{\eta,\lambda}^\prime); u).
\end{align*}

Then using $|A-B| \geq |A| - |B|$, we have
\begin{align*}
\Upsilon^{(1)} := &\sup_{u \in U} \Bigl \{ \Lambda^{+}(\widehat{\mathcal{G}}^{(1)}_{\psi} + \widehat{\mathcal{G}}^{(2)}_{\psi} + \widehat{\mathcal{G}}^{(3)}_{\psi}) -  \Lambda^{+}(\mathcal{G}^{(1)}_{\psi} + \mathcal{G}^{(2)}_{\psi} + \mathcal{G}^{(3)}_{\psi})  \Bigr \} \\
\leq & \sup_{u \in U} \Lambda^{+}(\widehat{\mathcal{G}}^{(1)}_{\psi} + \widehat{\mathcal{G}}^{(2)}_{\psi} + \widehat{\mathcal{G}}^{(3)}_{\psi} - (\mathcal{G}^{(1)}_{\psi} + \mathcal{G}^{(2)}_{\psi} + \mathcal{G}^{(3)}_{\psi}))\\
\leq &  \Upsilon^{(11)} + \Upsilon^{(12)} +  \Upsilon^{(13)},
\end{align*}
where
\begin{align*}
\Upsilon^{(11)} &:= \sup_{u \in U} \Lambda^{+} (\widehat{\mathcal{G}}^{(1)}_{\psi} - \mathcal{G}^{(1)}_{\psi}),~ \Upsilon^{(12)} := \sup_{u \in U} \Lambda^{+} (\widehat{\mathcal{G}}^{(2)}_{\psi} - \mathcal{G}^{(2)}_{\psi}) \\
\Upsilon^{(13)} &:= \sup_{u \in U} \Lambda^{+} (\widehat{\mathcal{G}}^{(3)}_{\psi} - \mathcal{G}^{(3)}_{\psi}).	
\end{align*}

The estimate of $\Upsilon^{(1i)}$, $i=1,2,3$, are obtained in (\ref{eq_5_18_3_2_3_2}), (\ref{eq_5_21_1_3_2_3_1}) and (\ref{eq_5_22_4_2_1_1}) separately below, which show that for any $\lambda \geq \max \{C_2, C_4\}$,  where $C_2$ and $C_4$ are given below,
\begin{align}
\label{eq_5_18_2_1_0_2_1_1}
\lim_{\kappa \rightarrow \infty} \Upsilon^{(1)} \leq \lim_{\kappa \rightarrow \infty} \{ \Upsilon^{(11)} + \Upsilon^{(12)} +  \Upsilon^{(13)} \} \leq 0.
\end{align}

\subsubsection{Estimate of $\Upsilon^{(11)}$} From definition,
\begin{align*}
\widehat{\mathcal{G}}_{\psi}^{(1)} &= \begin{bmatrix}
 	\widehat{q} + \frac{1}{2}\partial_t \widehat{h}_{\eta,\lambda}^\prime - d(\breve{a}^\prime,\Gamma) + \langle D_{(\breve{a},\breve{b})}(\widehat{h}_{\eta,\lambda}^\prime + \zeta_{\kappa}^\prime), \widehat{f}(t^\prime,\breve{a}^\prime,u) \rangle & 0 \\
 	0 & 0
 \end{bmatrix} \\
\mathcal{G}_{\psi}^{(1)} &= \begin{bmatrix}
 	-q - \frac{1}{2}\partial_t h_{\eta,\lambda}^\prime - d(a^\prime,\Gamma) - \langle D_{(a,b)}(h_{\eta,\lambda}^\prime + \zeta_{\kappa}^\prime), \widehat{f}(t^\prime,a^\prime,u) \rangle & 0 \\
 	0 & 0
 \end{bmatrix},
\end{align*}
which implies (note that $\widehat{q} + q = 0$)
\begin{align*}
\Upsilon^{(11)} & = \sup_{u \in U} \max \{\partial_t \widehat{h}_{\eta,\lambda}^\prime + \partial_t h_{\eta,\lambda}^\prime + \langle D_{(\breve{a},\breve{b})}(\widehat{h}_{\eta,\lambda}^\prime + \zeta_{\kappa}^\prime), \widehat{f}(t^\prime,\breve{a}^\prime,u) \rangle \\
&\qquad \qquad + \langle D_{(a,b)}(h_{\eta,\lambda}^\prime + \zeta_{\kappa}^\prime), \widehat{f}(t^\prime,a^\prime,u) \rangle + (d(a^\prime,\Gamma) - d(\breve{a}^\prime,\Gamma)) ,0 \}.
\end{align*}

We have
\begin{align*}
& \frac{1}{2}(\partial_t \widehat{h}_{\eta,\lambda}^\prime + \partial_t h_{\eta,\lambda}^\prime)	 \\
& =  -  \frac{\eta}{2} \lambda e^{-\lambda t^\prime}(1+|a^\prime|^2 + b^\prime) + (t-\tilde{t}) - \frac{\eta\lambda e^{-\lambda t^\prime} }{4}  \Bigl (|a^\prime - \tilde{a}|^2 + |b^\prime -\tilde{b}|^4 \Bigr ) \\
& \quad -  \frac{\eta}{2}  \lambda e^{-\lambda t^\prime}(1 + |\breve{a}^\prime|^2 + \breve{b}^\prime ) \\
& \rightarrow  - \eta  \lambda e^{-\lambda \tilde{t}}(1 + |\tilde{a}|^2 + \tilde{b})~ \text{as $\kappa \rightarrow \infty$ due to (\ref{eq_5_14})},
\end{align*}
and using Cauchy-Schwarz inequality, and Assumptions \ref{Assumption_1} and \ref{Assumption_2},
\begin{align*}
& \langle D_{(\breve{a},\breve{b})}(\widehat{h}_{\eta,\lambda}^\prime + \zeta_{\kappa}^\prime), \widehat{f}(t^\prime,\breve{a}^\prime,u) \rangle + \langle D_{(a,b)}(h_{\eta,\lambda}^\prime + \zeta_{\kappa}^\prime), \widehat{f}(t^\prime,a^\prime,u) \rangle	 \\
& \leq |2 \eta e^{-\lambda t^\prime}   \breve{a}^\prime - \kappa (a^\prime - \breve{a}^\prime)||1+|\breve{a}^\prime|| + |\eta e^{-\lambda t^\prime} - \kappa (b^\prime - \breve{b}^\prime)||1 + \breve{a}^\prime| \\
& \quad + |2 \eta e^{-\lambda t^\prime} a^\prime + \eta e^{-\lambda t^\prime} (a^\prime - \tilde{a}) 	 + \kappa (a^\prime - \breve{a}^\prime)| |1 + |a^\prime|| \\
& \quad + |(3/2) \eta e^{-\lambda t^\prime} + \kappa (b^\prime - \breve{b}^\prime)| |1 + |a^\prime|| \\
& \leq C_1 \eta e^{-\lambda t^\prime} ( 1 + |a^\prime|^2 + |\breve{a}^\prime|^2 + b^\prime + \breve{b}^\prime) \\
& \rightarrow  C_2 \eta e^{-\lambda \tilde{t}}(1 + |\tilde{a}|^2 + \tilde{b})~ \text{as $\kappa \rightarrow \infty$ due to (\ref{eq_5_14})}.
\end{align*}
Moreover, from Assumption \ref{Assumption_3},
\begin{align*}
	|d(a^\prime,\Gamma) - d(\breve{a}^\prime,\Gamma) | \leq C |a^\prime - \breve{a}^\prime| \rightarrow 0~\text{as $\kappa \rightarrow \infty$ due to (\ref{eq_5_14})}.
\end{align*}

Hence,
\begin{align*}
\lim_{\kappa \rightarrow \infty} \Upsilon^{(11)} \leq \max \{(-\lambda + C_2)\eta e^{-\lambda \tilde{t}}(1+|\tilde{a}|^2 + |\tilde{b}|^2),0 \},
\end{align*}
and for any $\lambda > 0$ with $\lambda \geq C_2$, we have
\begin{align}	
\label{eq_5_18_3_2_3_2}
\lim_{\kappa \rightarrow \infty} \Upsilon^{(11)} \leq 0.
\end{align}

\subsubsection{Estimate of $\Upsilon^{(12)}$}
From definition,
\begin{align*}
\widehat{\mathcal{G}}^{(2)}_{\psi} &= \begin{bmatrix}
 	\frac{1}{2} \Tr \bigl (\sigma \sigma^\top (t^\prime,\breve{a}^\prime,u) \widehat{P}_{(11)} \bigr ) & \frac{1}{2} \psi(\breve{b}^\prime) \widehat{P}_{(12)}^\top \sigma(t^\prime,\breve{a}^\prime,u) \\
 	\frac{1}{2} \psi(\breve{b}^\prime)\sigma^\top(t^\prime,\breve{a}^\prime,u) \widehat{P}_{(12)} & \frac{1}{2} \psi^2(\breve{b}^\prime) \widehat{P}_{(22)} I_r
 \end{bmatrix}	\\
\mathcal{G}^{(2)}_{\psi} &= \begin{bmatrix}
 	- \frac{1}{2} \Tr \bigl (\sigma \sigma^\top (t^\prime,a^\prime,u ) P_{(11)} \bigr  ) & - \frac{1}{2} \psi(b^\prime) P_{(12)}^\top \sigma(t^\prime,a^\prime,u) \\
 	- \frac{1}{2} \psi(b^\prime) \sigma^\top(t^\prime,a^\prime,u) P_{(12)} & - \frac{1}{2} \psi^2(b^\prime) P_{(22)} I_r
 \end{bmatrix}.
\end{align*}
Let
\begin{align*}
\Delta := \begin{bmatrix}
 	\sigma^\top (t^\prime, a^\prime,u) & 0 \\
 	0 & \psi(b^\prime)
 \end{bmatrix},~ \breve{\Delta}	:= \begin{bmatrix}
 	\sigma^\top (t^\prime, \breve{a}^\prime,u) & 0 \\
 	0 & \psi(\breve{b}^\prime)
 \end{bmatrix}.
\end{align*}
Using (\ref{eq_5_16_1_1_1_1}) and Assumption \ref{Assumption_1}, together with Cauchy-Schwarz inequality, we can show that for any $z \in \mathbb{R}^{r+1}$,
\begin{align}
\label{eq_5_18_1_1_1_1}
z^\top \begin{bmatrix}
 \Delta & \breve{\Delta}	
 \end{bmatrix}
	\begin{bmatrix}
		P & 0 \\
		0 & \widehat{P}
	\end{bmatrix} \begin{bmatrix}
 \Delta^\top \\
 \breve{\Delta}^\top	
 \end{bmatrix} z & \leq 3 \kappa z^\top \begin{bmatrix}
 \Delta & \breve{\Delta}	
 \end{bmatrix}
	\begin{bmatrix}
		I_{n+1} & -I_{n+1} \\
		-I_{n+1} & I_{n+1}
	\end{bmatrix} \begin{bmatrix}
 \Delta^\top \\
 \breve{\Delta}^\top	
 \end{bmatrix} z \\
 & = 3 \kappa |(\Delta - \breve{\Delta})^\top z|^2 \nonumber \\
 & \leq 3 \kappa \| \Delta^\top - \breve{\Delta} ^\top\|_{F}^2 |z|^2 \leq 3 \kappa C^2 (|a^\prime - \breve{a}^\prime|^2 + |b^\prime - \breve{b}^\prime|^2) |z|^2. \nonumber
 \end{align}
 
For $j \in \{1,\ldots,r\}$, let 
\begin{align*}
z_{(j)} := \begin{bmatrix}
 	\hat{z}_{(j)}^\top & z_j
 \end{bmatrix}^\top \in \mathbb{R}^{r+1},	
\end{align*}
where $z_j \in \mathbb{R}$ and $\hat{z}_{(j)}$ is an $r$-dimensional vector with $j$th entry being $\hat{z} \in \mathbb{R}$ and other entries being zero, i.e., $\hat{z}_{(j)} := \begin{bmatrix}
 0 & \cdots & 0 & \hat{z} & 0 & \cdots & 0 
 \end{bmatrix}
$. Then
\begin{align}
\label{eq_5_19_1111_1_1_1}
& \frac{1}{2} z_{(j)}^\top	(\Delta P \Delta^\top + \breve{\Delta} \widehat{P} \breve{\Delta}^\top) z_{(j)} \\
& = \frac{1}{2} \hat{z}^2  \Bigl ( \sigma^\top (t^\prime,\breve{a}^\prime,u) \widehat{P}_{(11)} \sigma (t^\prime,\breve{a}^\prime,u)  + \sigma^\top (t^\prime,a^\prime,u) P_{(11)}  \sigma  (t^\prime,a^\prime,u) \Bigr )_{jj} \nonumber \\
& \quad + \hat{z} \Bigl ( \psi(\breve{b}^\prime) \widehat{P}_{(12)}^\top \sigma(t^\prime,\breve{a}^\prime,u) + \psi(b^\prime) P_{(12)}^\top \sigma(t^\prime,a^\prime,u) \Bigr )_{j} z_{j}    \nonumber \\
& \quad + \frac{1}{2} z_{j}^2 \Bigl ( \psi^2(\breve{b}^\prime) \widehat{P}_{(22)} +  \psi^2(b^\prime) P_{(22)}\Bigr )  \nonumber \\
& \leq \frac{3}{2} \kappa C^2 (|a^\prime - \breve{a}^\prime|^2 + |b^\prime - \breve{b}^\prime|^2) (\hat{z}_1^2 +z_j^2), \nonumber
\end{align}
where the inequality follows from (\ref{eq_5_18_1_1_1_1}).\footnote{In (\ref{eq_5_19_1111_1_1_1}) and below, $(\cdot)_{j}$ and $(\cdot)_{jj}$ indicate the $j$th component of the vector, and the $j$th element of the row and column of the matrix, respectively.}

Let
\begin{align*}
y := \begin{bmatrix}
 \hat{z} & y_2^\top	
 \end{bmatrix}^\top,~ y_2 := \begin{bmatrix}
 	z_1 & \cdots & z_r
 \end{bmatrix}^\top.
\end{align*}
Using (\ref{eq_5_19_1111_1_1_1}), we can show that
\begin{align*}
& y^\top (	\widehat{\mathcal{G}}^{(2)}_{\psi} - \mathcal{G}^{(2)}_{\psi} ) y \\
& = \frac{1}{2} \hat{z}^2 \Bigl (\Tr \bigl (\sigma \sigma^\top (t^\prime,\breve{a}^\prime,u) \widehat{P}_{(11)} \bigr ) + \Tr (\sigma \sigma^\top (t^\prime,a^\prime,u) P_{(11)} ) \Bigr ) \\
& \quad + \hat{z} \Bigl ( \psi(\breve{b}^\prime) \widehat{P}_{(12)}^\top \sigma(t^\prime,\breve{a}^\prime,u) + \psi(b^\prime) P_{(12)}^\top \sigma(t^\prime,a^\prime,u) \Bigr ) y_2 \\
& \quad + \frac{1}{2} y_2^\top \Bigl ( \psi^2(\breve{b}^\prime) \widehat{P}_{(22)} I_r +  \psi^2(b^\prime) P_{(22)} I_r \Bigr ) y_2 \\
& = \frac{1}{2} \sum_{j=1}^r \hat{z}^2  \Bigl ( \sigma^\top (t^\prime,\breve{a}^\prime,u) \widehat{P}_{(11)} \sigma (t^\prime,\breve{a}^\prime,u)  + \sigma^\top (t^\prime,a^\prime,u) P_{(11)}  \sigma  (t^\prime,a^\prime,u) \Bigr )_{jj} \\
& \quad + \sum_{j=1}^r \hat{z} \Bigl ( \psi(\breve{b}^\prime) \widehat{P}_{(12)}^\top \sigma(t^\prime,\breve{a}^\prime,u) + \psi(b^\prime) P_{(12)}^\top \sigma(t^\prime,a^\prime,u) \Bigr )_{j} z_{j} \\
& \quad + \frac{1}{2} \sum_{j=1}^r z_{j}^2 \Bigl ( \psi^2(\breve{b}^\prime) \widehat{P}_{(22)} +  \psi^2(b^\prime) P_{(22)}\Bigr ) \\
& \leq \frac{3}{2} \kappa C^2 (|a^\prime - \breve{a}^\prime|^2 + |b^\prime - \breve{b}^\prime|^2)  \sum_{i=1}^r (\hat{z}^2 +z_j^2) \\
& \leq \frac{3}{2} \kappa C^2 (|a^\prime - \breve{a}^\prime|^2 + |b^\prime - \breve{b}^\prime|^2) r |y|^2,
\end{align*}
which, together with the arbitrariness of $\hat{z}$ and $z_j$, $j \in \{1,\ldots,r\}$, leads to
\begin{align*}
\max_{|y|=1} y^\top (	\widehat{\mathcal{G}}^{(2)}_{\psi} - \mathcal{G}^{(2)}_{\psi} ) y \leq \frac{3}{2} r \kappa C^2 (|a^\prime - \breve{a}^\prime|^2 + |b^\prime - \breve{b}^\prime|^2).
\end{align*}
Hence, in view of (\ref{eq_5_14}) and the definition of $\Lambda^+$ (see Lemma \ref{Lemma_9} and \cite[Example 5.6.6]{Horn_book_2013}), we have
\begin{align}
\label{eq_5_21_1_3_2_3_1}
\lim_{\kappa \rightarrow \infty} \Upsilon^{(12)} \leq 0.
\end{align}

\subsubsection{Estimate of $\Upsilon^{(13)}$}
By definition, we have
\begin{align*}
\widehat{\mathcal{G}}^{(3)}_{\psi} &= \begin{bmatrix}
 	\frac{1}{2}\partial_t \widehat{h}_{\eta,\lambda}^\prime + \eta e^{-\lambda t^\prime} \Tr \bigl (\sigma \sigma^\top (t^\prime,\breve{a}^\prime,u)\bigr ) & 0 \\
0 & 0_{r \times r}
 \end{bmatrix}	\\
\mathcal{G}^{(3)}_{\psi} &= \begin{bmatrix}
 	- \frac{1}{2}\partial_t h_{\eta,\lambda}^\prime  - \frac{3}{2} \eta e^{-\lambda t^\prime} \Tr \bigl (\sigma \sigma^\top (t^\prime,a^\prime,u )  \bigr  ) & 0 \\
 	0 & 0_{r \times r}
 \end{bmatrix},
\end{align*}
which implies
\begin{align*}
& \Upsilon^{(13)} = \sup_{u \in U} \max \{\frac{1}{2}(\partial_t \widehat{h}_{\eta,\lambda}^\prime + \partial_t h_{\eta,\lambda}^\prime) + \eta e^{-\lambda t^\prime} \Tr \bigl (\sigma \sigma^\top (t^\prime,\breve{a}^\prime,u)\bigr ) \\
& \qquad \qquad  + \frac{3}{2}\eta e^{-\lambda t^\prime} \Tr \bigl (\sigma \sigma^\top (t^\prime,a^\prime,u )  \bigr  ),0\}.
\end{align*}

Note that
from Assumption \ref{Assumption_1}, 
\begin{align*}
& \Bigl | \eta e^{-\lambda t^\prime} \Tr \bigl (\sigma \sigma^\top (t^\prime,\breve{a}^\prime,u)\bigr ) + \frac{3 \eta e^{-\lambda t^\prime}}{2} \Tr \bigl (\sigma \sigma^\top (t^\prime,a^\prime,u )  \bigr  ) \Bigr | \\
& = | \eta e^{-\lambda t^\prime}  \|\sigma(t^\prime,\breve{a}^\prime,u) \|_{F}^2 + \frac{3 \eta e^{-\lambda t^\prime}}{2} \| \sigma(t^\prime,a^\prime,u )\|_{F}^2 |
\\
& \leq C_3 \eta e^{-\lambda t^\prime} (1 + |a^\prime |^2 + |\breve{a}^\prime|^2) \\
& \rightarrow C_4 \eta e^{- \lambda \tilde{t}} (1 + |\tilde{a}|^2)~ \text{as $\kappa \rightarrow \infty$ due to (\ref{eq_5_14})},
\end{align*}
and as shown above, 
\begin{align*}
& \frac{1}{2}(\partial_t \widehat{h}_{\eta,\lambda}^\prime + \partial_t h_{\eta,\lambda}^\prime)	 
\rightarrow  - \eta  \lambda e^{-\lambda \tilde{t}}(1 + |\tilde{a}|^2 + \tilde{b})~ \text{as $\kappa \rightarrow \infty$ due to (\ref{eq_5_14})}.
\end{align*}

Hence,
\begin{align*}
\lim_{\kappa \rightarrow \infty} \Upsilon^{(13)} & \leq \max \{(C_4 - \lambda) \eta   e^{-\lambda \tilde{t}}(1 + |\tilde{a}|^2 + \tilde{b}), 0\},
\end{align*}
and if we choose $\lambda > 0$ with $\lambda \geq C_4$, then
\begin{align}
\label{eq_5_22_4_2_1_1}
\lim_{\kappa \rightarrow \infty} \Upsilon^{(13)} & \leq 0.
\end{align}

\subsection{Estimate of $\Upsilon^{(2)}$}
In view of the definition of $H^{(21)}$,
\begin{align*}
\Upsilon^{(2)} = \sup_{(u,\beta(e)) \in U \times G^2} \{	\Upsilon^{(21)} + \Upsilon^{(22)} \},
\end{align*}
where
\begin{align*}
	\Upsilon^{(21)} &:= -\int_{E_{\delta}} [ \phi^\prime(t^\prime, \breve{a}^\prime + \chi(t^\prime, \breve{a}^\prime,u,e), \breve{b}^\prime + \beta(e)) - \phi^\prime(t^\prime, \breve{a}^\prime, \breve{b}^\prime) ]\pi( \dd e) 
 \\
 &\qquad \qquad + \int_{E_{\delta}}  \langle D \phi^\prime(t^\prime, \breve{a}^\prime, \breve{b}^\prime),  \begin{bmatrix}
 	\chi(t^\prime, \breve{a}^\prime,u,e) \\
 	\beta (e)
 \end{bmatrix} \rangle \pi(\dd e) \\
\Upsilon^{(22)} &:= \int_{E_{\delta}} [ \phi(t^\prime,a^\prime + \chi(t^\prime,a^\prime,u,e), b^\prime + \beta(e)) - \phi(t^\prime,a^\prime,b^\prime) ]\pi( \dd e) 
 \\
 &\qquad \qquad \qquad - \int_{E_{\delta}}  \langle D \phi(t,a^\prime,b^\prime),  \begin{bmatrix}
 	\chi(t^\prime,a^\prime,u,e) \\
 	\beta (e)
 \end{bmatrix} \rangle \pi(\dd e).
\end{align*}

Let $\chi^\prime(u,e) := \chi(t^\prime,a^\prime,u,e)$. From Lemma \ref{Lemma_B_1} in Appendix \ref{Appendix_B} and H\"oder inequality, it follows from the uniform boundedness of $D^2\phi$ that
\begin{align*}
\Upsilon^{(22)} 
= & \int_{E_{\delta}} \int_0^1 (1-z) \Tr \Bigl ( D^2 \phi (t^\prime,a^\prime + z \chi^\prime(u,e), b^\prime + z \beta(e)) \\
& \qquad \qquad \times \begin{bmatrix}
 	\chi^\prime (\chi^\prime)^\top(u,e) & \chi^\prime(u,e) \beta^\top(e) \\
 	\beta (e) (\chi^\prime)^\top(u,e) & \beta(e) \beta^\top(e) \nonumber
 \end{bmatrix} \Bigr ) \dd z  \pi (\dd e) \\
 \leq & \int_{E_{\delta}} \int_0^1 (1-z) \Bigl \| D^2 \phi (t^\prime,a^\prime + z \chi^\prime(u,e), b^\prime + z \beta(e)) \Bigr \|_{F} \nonumber \\
& \qquad \qquad \times (|\chi^\prime(u,e)| + |\beta(e)|) \dd z \pi (\dd e).  \\
&  \leq C \Bigl ( \bigl (\int_{E_{\delta}} |\chi^\prime(u,e)|^2 \pi (\dd e) \bigr )^{\frac{1}{2}} + \bigl ( \int_{E_{\delta}}  |\beta(e)|^2 \pi( \dd e) \bigr)^{\frac{1}{2}} \Bigr ).
\end{align*}
Then the regularity of $\chi$ in Assumption \ref{Assumption_1} and the fact that $\beta \in G^2(E,\mathcal{B}(E),\pi;\mathbb{R})$ can be restricted to a uniformly bounded control from Remark \ref{Remark_2} imply that $\lim_{\delta \downarrow 0} \Upsilon^{(22)} \leq 0$. A similar technique can be applied to show that $\lim_{\delta \downarrow 0} \Upsilon^{(21)} \leq 0$. 

Hence, we have
\begin{align}
\label{eq_5_25_3_2_6_2_1_2}
	\lim_{\delta \downarrow 0} \Upsilon^{(2)} \leq 0.
\end{align}

\subsection{Estimate of $\Upsilon^{(3)}$}
Recall (\ref{eq_5_15_111})
\begin{align*}
\Psi_{\nu;\eta,\lambda}^{\kappa} (t,a,\breve{a},b,\breve{b}) =& (\underline{W}_{\nu}(t,a,b) - h_{\eta,\lambda}(t,a,b)) \\
& - (\overline{W}(t,\breve{a},\breve{b}) + \widehat{h}(t,\breve{a},\breve{b})) - \zeta_{\kappa}(a,b,\breve{a},\breve{b}), \nonumber
\end{align*}
from which we have
\begin{align}	
\label{eq_5_17_1_1_1}
& \underline{W}_{\nu}(t,a,b)  - \overline{W}(t,\breve{a},\breve{b}) \\
& = \Psi_{\nu;\eta,\lambda}^{\kappa} 	(t,a,b,\breve{a},\breve{b}) +  h_{\eta,\lambda}(t,a,b) + \widehat{h}_{\eta,\lambda}(t,\breve{a},\breve{b}) + \zeta_{\kappa}(a,b,\breve{a},\breve{b}). \nonumber
\end{align}
We note that $(t^\prime, a^\prime, b^\prime,\breve{a}^\prime,\breve{b}^\prime)$ is the maximum point of $\Psi_{\nu;\eta,\lambda}^{\kappa}$.

Let $\chi^\prime(u,e) := \chi(t^\prime,a^\prime,u,e)$ and  $\breve{\chi}^\prime(u,e) := \chi(t^\prime,\breve{a}^\prime,u,e)$. Since $(t^\prime, a^\prime, b^\prime,\breve{a}^\prime,\breve{b}^\prime)$ is the maximum point of $\Psi_{\nu;\eta,\lambda}^{\kappa}$, it follows from (\ref{eq_5_17_1_1_1}) and the definition of $\Upsilon^{(3)}$ that
\begin{align*}
\Upsilon^{(3)} 
& = \sup_{u \in U, \beta \in G^2 }  \Bigl \{ 
\int_{E_{\delta}^C} \Bigl [ \Psi_{\nu;\eta,\lambda}^{\kappa}(t^\prime, a^\prime + \chi^\prime(u,e), b^\prime + \beta(e), \breve{a}^\prime + \breve{\chi}^\prime(u,e), \breve{b}^\prime + \beta(e)) \\
&\qquad \qquad - \Psi_{\nu;\eta,\lambda}^{\kappa}(t^\prime,a^\prime, b^\prime, \breve{a}^\prime, \breve{b}^\prime) +  h_{\eta,\lambda}(t^\prime, a^\prime + \chi^\prime(u,e),b^\prime + \beta(e)) \\
& \qquad \qquad  + \widehat{h}_{\eta,\lambda}(t^\prime,\breve{a}^\prime + \breve{\chi}^\prime(u,e), \breve{b}^\prime + \beta(e)) \\
& \qquad \qquad + \zeta_{\kappa}(a^\prime + \chi^\prime(u,e), b^\prime + \beta(e), \breve{a}^\prime + \breve{\chi}^\prime(u,e), \breve{b}^\prime + \beta(e)) \\
& \qquad \qquad - (h_{\eta,\lambda}(t^\prime,a^\prime,b^\prime) + \widehat{h}_{\eta,\lambda}(t^\prime,\breve{a}^\prime,\breve{b}^\prime) + \zeta_{\kappa}(a^\prime,b^\prime,\breve{a}^\prime,\breve{b}^\prime)) \Bigr ] \pi (\dd e) \\
&\qquad \qquad \qquad+ \int_{E_{\delta}^C} \langle - D_{(\breve{a},\breve{b})}( \widehat{h}_{\eta,\lambda}^\prime + \zeta_{\kappa}^\prime), \begin{bmatrix}
 	\breve{\chi}^\prime(u,e) \\
 	\beta (e)
 \end{bmatrix} \rangle \pi(\dd e) \\
&\qquad \qquad \qquad - \int_{E_{\delta}^C} \langle D_{(a,b)}(h_{\eta,\lambda}^\prime + \zeta^\prime_{\kappa}), \begin{bmatrix}
 	\chi^\prime(u,e) \\
 	\beta (e)
 \end{bmatrix} \rangle \pi(\dd e)  \Bigr \} \\
& \leq \sup_{u \in U, \beta \in G^2 } \{ \Upsilon^{(31)} + \Upsilon^{(32)} + \Upsilon^{(33)} \},
 \end{align*}
where
\begin{align*}
\Upsilon^{(31)} & := \int_{E_{\delta}^C} \Bigl [h_{\eta,\lambda}(t^\prime, a^\prime + \chi^\prime(u,e),b^\prime + \beta(e)) - h_{\eta,\lambda}(t^\prime,a^\prime,b^\prime) \Bigr ] \pi (\dd e) \\
& \quad - \int_{E_{\delta}^C} \langle D_{(a,b)} h_{\eta,\lambda}^\prime, \begin{bmatrix}
 	\chi^\prime(u,e) \\
 	\beta (e)
 \end{bmatrix} \rangle \pi(\dd e) \\
\Upsilon^{(32)} & := \int_{E_{\delta}^C} \Bigl [\widehat{h}_{\eta,\lambda}(t^\prime,\breve{a}^\prime + \breve{\chi}^\prime(u,e), \breve{b}^\prime + \beta(e)) - \widehat{h}_{\eta,\lambda}(t^\prime,\breve{a}^\prime,\breve{b}^\prime) \Bigr ] \pi (\dd e) \\
& \quad - \int_{E_{\delta}^C} \langle D_{(\breve{a},\breve{b})}\widehat{h}_{\eta,\lambda}^\prime, \begin{bmatrix}
 	\breve{\chi}^\prime(u,e) \\
 	\beta (e)
 \end{bmatrix} \rangle \pi(\dd e) \\ 
 \Upsilon^{(33)} &:= \int_{E_{\delta}^C} \Bigl [ \zeta_{\kappa}(a^\prime + \chi^\prime(u,e), b^\prime + \beta(e), \breve{a}^\prime + \breve{\chi}^\prime(u,e), \breve{b}^\prime + \beta(e))  \\
& \qquad \qquad - \zeta_{\kappa}(a^\prime,b^\prime,\breve{a}^\prime,\breve{b}^\prime)) \Bigr ] \pi (\dd e) \\
& \quad -  \int_{E_{\delta}^C}  \langle D_{(a,b)} \zeta^\prime_{\kappa}, \begin{bmatrix}
 	\chi^\prime(u,e) \\
 	\beta (e)
 \end{bmatrix} \rangle \pi(\dd e) + \int_{E_{\delta}^C} \langle - D_{(\breve{a},\breve{b})} \zeta_{\kappa}^\prime, \begin{bmatrix}
 	\breve{\chi}^\prime(u,e) \\
 	\beta (e)
 \end{bmatrix} \rangle \pi(\dd e).	
\end{align*}

From Lemma \ref{Lemma_B_1} in Appendix \ref{Appendix_B} and (\ref{eq_5_15_1_1_1}),
\begin{align}
\label{eq_5_22_1_1_1}
\Upsilon^{(31)} & = \int_{E_{\delta}^C} \int_0^1 (1-z) \Tr \Bigl ( \begin{bmatrix}
3 \eta e^{-\lambda t^\prime} I_n & 0  	\\
0 & 0
 \end{bmatrix} \\
& \qquad \qquad \times \begin{bmatrix}
 	\chi^\prime (\chi^\prime)^\top(u,e) & \chi^\prime(u,e) \beta^\top(e) \\
 	\beta (e) (\chi^\prime)^\top(u,e) & \beta(e) \beta^\top(e) \nonumber
 \end{bmatrix} \Bigr ) \dd z  \pi (\dd e) \nonumber \\
 & \leq Cn \eta e^{-\lambda t^\prime} (1 + |a^\prime|^2), \nonumber
 \end{align}
 and similarly,
 \begin{align}
\label{eq_5_22_1_1_1_4_3_2_1}
 \Upsilon^{(31)} & = \int_{E_{\delta}^C} \int_0^1 (1-z) \Tr \Bigl ( \begin{bmatrix}
2 \eta e^{-\lambda t^\prime}   I_n & 0 \\
0 & 0
 \end{bmatrix} \\
& \qquad \qquad \times \begin{bmatrix}
 	\breve{\chi}^\prime (\breve{\chi}^\prime)^\top(u,e) & \breve{\chi}^\prime(u,e) \beta^\top(e) \\
 	\beta (e) (\breve{\chi}^\prime)^\top(u,e) & \beta(e) \beta^\top(e) \nonumber
 \end{bmatrix} \Bigr ) \dd z  \pi (\dd e) \nonumber \\
 & \leq Cn \eta e^{-\lambda t^\prime}(1 + |\breve{a}^\prime|^2). \nonumber
 \end{align}
Moreover, using (\ref{eq_5_15_1_1_1}) and Assumption \ref{Assumption_1},
\begin{align}
\label{eq_5_23_1_2_1_2_1}
\Upsilon^{(33)} & = \frac{\kappa}{2} \int_{E_{\delta}^C} |\chi^\prime(u,e) - \breve{\chi}^\prime(u,e)|^2 \pi (\dd e) \leq \frac{\kappa}{2} |a^\prime - \breve{a}^\prime|^2 \\
& \rightarrow 0~ \text{as $\kappa \rightarrow \infty$ due to (\ref{eq_5_14})}. \nonumber 
\end{align}

Hence, (\ref{eq_5_22_1_1_1})-(\ref{eq_5_23_1_2_1_2_1}), together with (\ref{eq_5_14}), imply that
\begin{align}
\label{eq_5_30_2_1_2_1}
\lim_{\eta \downarrow 0} \lim_{\kappa \rightarrow \infty} \lim_{\delta \downarrow 0} \Upsilon^{(3)} \leq 0.
\end{align}

\appendix
\numberwithin{equation}{section}
\renewcommand{\theequation}{\thesection.\arabic{equation}}

\section{Proof of Lemma \ref{Lemma_1}}\label{Appendix_A}
\begin{proof}[Proof of Lemma \ref{Lemma_1}] 
For the existence and uniqueness in (i), see \cite[Theorem 6.2.3]{Applebaum_book} (\cite[Theorem 1.19]{Oksendal_book_jump} and \cite{Fujiwara_Kyoto_1985}). Let $f(x) := f(t,x,u)$, $\sigma(x) := \sigma(t,x,u)$ and $\chi(x) := \chi(t,x,u,e)$. Then note that
\begin{align*}
|x_s^{t,a;u} - x_s^{t,a^\prime;u}|^2 \leq & 4 |a-a^\prime|^2 + 4 \Bigl | \int_{t}^{s} |f(x_s^{t,a;u}) - f(x_s^{t,a^\prime;u})| \dd s \Bigr |^2 \\
& + 4 \Bigl | \int_t^s [\sigma(x_s^{t,a;u}) - \sigma (x_s^{t,a^\prime;u}) ]\dd B_s \Bigr |^2 \\
&+ 4 \Bigl | \int_{t}^{s} \int_{E} [\chi(x_s^{t,a;u}) - \chi(x_s^{t,a^\prime;u})] \tilde{N}(\dd e, \dd s) \Bigr |^2.
\end{align*}
By H\"older inequality and Assumption \ref{Assumption_1},
\begin{align}
\label{eq_a_1}
\mathbb{E} \Bigl [ \Bigl |\int_{t}^{s} |f(x_r^{t,a;u}) - f(x_r^{t,a^\prime;u})| \dd r \Bigr |^2 \Bigr ] & \leq C \mathbb{E} \int_t^s |x_r^{t,a;u} - x_r^{t,a^\prime;u}|^2 \dd r,
\end{align}
and applying Burkholder-Davis-Gundy inequality \cite[Theorem 4.4.21]{Applebaum_book} and Assumption \ref{Assumption_1} yields
\begin{align}
\label{eq_a_2}
\mathbb{E} \Bigl [ \Bigl | \int_t^s [\sigma(x_r^{t,a;u}) - \sigma (x_r^{t,a^\prime;u}) ]\dd B_r \Bigr |^2	\Bigr ] \leq C \mathbb{E} \int_t^s |x_r^{t,a;u} - x_r^{t,a^\prime;u}|^2 \dd r.
\end{align}
Moreover, from Kunita's formula for general L\'evy-type stochastic integrals \cite[Theorem 4.4.23]{Applebaum_book} and Assumption \ref{Assumption_1},
\begin{align}	
\label{eq_a_3}
& \mathbb{E} \Bigl [ \Bigl | \int_{t}^{s} \int_{E} [\chi(x_r^{t,a;u}) - \chi(x_r^{t,a^\prime;u})] \tilde{N}(\dd e, \dd r) \Bigr |^2 \Bigr ] \\
& \leq C \mathbb{E} \Bigl [  \int_t^s \int_{E} |\chi(x_r^{t,a;u}) - \chi(x_r^{t,a^\prime;u})|^{2} \pi (\dd e) \dd r \Bigr ]  \leq C \mathbb{E} \int_t^s |x_r^{t,a;u} - x_r^{t,a^\prime;u}|^2 \dd r. \nonumber
\end{align}
Then using (\ref{eq_a_1})-(\ref{eq_a_3}), together with Gronwall's lemma, we get (\ref{eq_1_2}). The proof for (\ref{eq_1_1}) is analogous, for which we have to use the linear growth condition in Assumption \ref{Assumption_1}.

To prove (\ref{eq_1_3}), note that
\begin{align*}
x_s^{t,a;u} = x_s^{t^\prime, x_{t^\prime}^{t,a;u};u},~ \forall s \in [t^\prime,T].	
\end{align*}
Then using (\ref{eq_1_2}) and the approach similar to above, it follows that
\begin{align*}
	\mathbb{E} \Bigl [\sup_{s \in [t^\prime,T]} | x_s^{t^\prime, x_{t^\prime}^{t,a;u};u} - x_s^{t^\prime,a;u} |^2 \Bigr ] \leq C \mathbb{E} \Bigl [ |x_{t^\prime}^{t,a;u} - a|^2 \Bigr ] \leq C(1+|a|^2)|t^\prime-t|.
\end{align*}
This completes the proof.
\end{proof}

\section{Technical Lemma} \label{Appendix_B}
The following lemma is given in \cite[Lemma 4.3, Chapter 3]{Yong_book} without its proof. Here, we provide a complete proof.
\begin{lemma}\label{Lemma_B_1}	
Suppose that $g \in C^2(\mathbb{R}^n)$. Then for any $x,a \in \mathbb{R}^n$,
\begin{align*}
g(x + a) = g(x) + \langle D g(x), a \rangle + \int_0^1 (1-z) \langle D^2 g (x + z a) a, a \rangle \dd z.
\end{align*}
\end{lemma}
\begin{proof}
Note that $\frac{\dd }{ \dd z} g(x + za) = \langle D g(x + za), a \rangle$,
which leads to
\begin{align*}
g(x + a) - g(x) = \int_0^1 \langle D g(x+ za), a \rangle \dd z  = \int_0^1 \sum_{i=1}^n \partial_{x_i} g (x +  z a) a_i \dd z.
\end{align*}
Using the integration by parts formula $\int_0^1 u \frac{ \dd v}{\dd z}\dd z = - \int_0^1  v \frac{\dd u}{\dd z} \dd z + uv|_{0}^{1}$ with $u=\sum_{i=1}^n \partial_{x_i} g (x +  z a) a_i$ and $v=z-1$ yields (note that $\frac{\dd u}{\dd z} = \sum_{i,j=1}^n \partial_{x_i x_j} g (x + za) a_i a_j$)
\begin{align*}
g(x + a) - g(x) &= \sum_{i=1}^n \partial_{x_i} g (x) a_i + \int_0^1 (1-z) \sum_{i,j=1}^n  \partial_{ x_i x_j} g (x + za) a_i a_j \dd z \\
& = \langle D g(x), a \rangle + \int_0^1 (1-z) \langle D^2 g (x + z a) a, a \rangle \dd z.
\end{align*}
We complete the proof.
\end{proof}

\section{Existence of Optimal Controls for Jump Diffusion Systems}\label{Appendix_C}

In Theorem \ref{Theorem_1}, an additional assumption of the existence of optimal controls for the auxiliary optimal control problem in (\ref{eq_10}) is required. Here, we show that a certain class of stochastic optimal control problems for jump diffusion systems with unbounded control sets admits an optimal control. The proof of the main result in this appendix (see Theorem \ref{Theorem_C_1}) extends the case of SDEs in a Brownian setting without jumps studied in \cite[Appendix A]{Bokanowski_SICON_2016} and \cite[Theorem 5.2, Chapter 2]{Yong_book} to the framework of jump diffusion systems.

As in (\ref{eq_10}), consider
\begin{align}
\label{eq_c_1}
W(t,a,b) & :=  \mathop{\inf_{u \in \mathcal{U}}}_{\alpha \in \mathcal{A},\beta \in \mathcal{B}} \overline{J}(t,a,b;u,\alpha,\beta),
\end{align}
where
\begin{align*}
\overline{J}(t,a,b;u,\alpha,\beta) = \mathbb{E} \Bigl [ \rho_2(x_T^{t,a;u}, y_{T;t,a,b}^{u,\alpha,\beta})  + \int_t^T \rho_1(s,x_s^{t,a;u}, u_s) \dd s \Bigr ],
\end{align*}
and subject to (we recall (\ref{eq_1}) and (\ref{eq_5}))
\begin{align*}
& \begin{cases}
\dd x_{s}^{t,a;u} = f(s,x_{s}^{t,a;u},u_s)\dd s + \sigma(s,x_{s}^{t,a;u},u_s) \dd B_s	\\
\qquad \qquad + \int_{E} \chi(s,x_{s-}^{t,a;u},u_s,e) \tilde{N}(\dd e, \dd s), ~ x_{t}^{t,a;u} = a \\
\dd y_{s; t,a,b}^{u,\alpha,\beta} = - 	l(s,x_{s}^{t,a;u},u_s) \dd s + \alpha_s^\top \dd B_s + \int_{E} \beta_s(e)  \tilde{N}(\dd e, \dd s),~ y_{t; t,a,b}^{u,\alpha,\beta} = b.
\end{cases}
\end{align*}

\begin{assumption}\label{Assumption_4}
\begin{enumerate}[(i)]
\item For $\iota := f, \sigma, \chi, l$ with $\iota = \begin{bmatrix} \iota_1^\top & \cdots & \iota_n^\top\end{bmatrix}^\top$, $\iota$ satisfies Assumptions \ref{Assumption_1} and \ref{Assumption_2}, and is independent of $x$. Moreover, $\iota_i$, $i=1,\ldots,n$, is convex and Lipschitz continuous in $u$ with the Lipschitz constant $L$;
\item $\rho_1$ and $\rho_2$ are convex, nondecreasing and bounded from below;
\item $U \subset \mathbb{R}^m$ is a compact and convex set.
\end{enumerate}	
\end{assumption}
Note that Assumption \ref{Assumption_4} is different from that in \cite[Appendix A]{Bokanowski_SICON_2016} and \cite[Theorem 5.2, Chapter 2]{Yong_book}. We have the following result:
\begin{theorem}\label{Theorem_C_1}
Suppose that Assumption \ref{Assumption_4} holds. Then (\ref{eq_c_1}) admits an optimal solution $(\widehat{u}, \widehat{\alpha}, \widehat{\beta}) \in \mathcal{U} \times \mathcal{A} \times \mathcal{B}$, i.e., 
\begin{align*}
W(t,a,b) = \overline{J}(t,a,b;\widehat{u}, \widehat{\alpha}, \widehat{\beta}) = \mathop{\inf_{u \in \mathcal{U}}}_{\alpha \in \mathcal{A},\beta \in \mathcal{B}} \overline{J}(t,a,b;u,\alpha,\beta).	
\end{align*}
\end{theorem}
\begin{proof}
Since $\rho_1$ and $\rho_2$ are bounded from below, (\ref{eq_c_1}) is well defined. Suppose that $\{(\widehat{u}_k, \widehat{\alpha}_k, \widehat{\beta}_k)\}_{k \geq 1} \in \mathcal{U} \times \mathcal{A} \times \mathcal{B}$ is a sequence of minimizing controllers such that
\begin{align*}
\overline{J}(t,a,b;\widehat{u}_k, \widehat{\alpha}_k, \widehat{\beta}_k) \xrightarrow{k \rightarrow \infty} W(t,a,b).
\end{align*}
Note that $\mathcal{L}_{\mathbb{F}}^2$ and $\mathcal{G}_{\mathbb{F}}^2$ are Hilbert spaces. Also, from Remark \ref{Remark_2}, $\{(\widehat{\alpha}_k,\widehat{\beta}_k)\}_{k \geq 1}$ can be restricted to a sequence of uniformly bounded controls in $\mathcal{L}_{\mathbb{F}}^2$ and $\mathcal{G}_{\mathbb{F}}^2$ senses, and $U$ is compact from (iii) of Assumption \ref{Assumption_4}. Hence, in view of \cite[Theorem 3.18]{Brezis_book}, we can extract a subsequence $\{(u_{k_i},\widehat{\alpha}_{k_i},\widehat{\beta}_{k_i})\}_{i \geq 1}$ from $\{(\widehat{u}_k, \widehat{\alpha}_k, \widehat{\beta}_k) \}_{k \geq 1}$ such that
\begin{align*}
	(\widehat{u}_{k_i},\widehat{\alpha}_{k_i},\widehat{\beta}_{k_i}) \xrightarrow{i \rightarrow \infty} (\widehat{u}, \widehat{\alpha},\widehat{\beta})~ \text{weakly in $\mathcal{L}_{\mathbb{F}}^2 \times \mathcal{L}_{\mathbb{F}}^2 \times \mathcal{G}_{\mathbb{F}}^2$}.
\end{align*}
Then for each $\epsilon > 0$, there exists $i^\prime$ such that for any $i \geq i^\prime$,
\begin{align}
\label{eq_c_2_1_1_1}
\overline{J}(t,a,b,;\widehat{u}_{k_i},\widehat{\alpha}_{k_i},\widehat{\beta}_{k_i}) \leq W(t,a,b) + \frac{\epsilon}{2}.	
\end{align}

From Mazur's lemma \cite[Corollary 3.8]{Brezis_book}, we have convex combinations of subsequences above
\begin{align}
\label{eq_c_2}
	(\widetilde{u}_{k_i},\widetilde{\alpha}_{k_i},\widetilde{\beta}_{k_i}) &:= \sum_{p \geq 1} \theta_{k_i p}(\widehat{u}_{k_i + p},\widehat{\alpha}_{k_i + p},\widehat{\beta}_{k_i + p}),~ \theta_{k_i p} \geq 0,~ \sum_{p\geq 1} \theta_{k_i p} = 1,
\end{align}
such that
\begin{align}
\label{eq_c_4_1_1_1}
	(\widetilde{u}_{k_i},\widetilde{\alpha}_{k_i},\widetilde{\beta}_{k_i}) \xrightarrow{i \rightarrow \infty} (\widehat{u}, \widehat{\alpha},\widehat{\beta})~ \text{strongly in $\mathcal{L}_{\mathbb{F}}^2 \times \mathcal{L}_{\mathbb{F}}^2 \times \mathcal{G}_{\mathbb{F}}^2$},
\end{align}
where $(\widehat{u}, \widehat{\alpha},\widehat{\beta}) \in \mathcal{U} \times \mathcal{A} \times \mathcal{B}$. 

Then from (\ref{eq_c_2}) and (i) of Assumption \ref{Assumption_4}, we have
\begin{align*}
x_{s}^{t,a;\widetilde{u}_{k_i}} \preceq \sum_{p \geq 1} \theta_{k_ip} x_{s}^{t,a;\widehat{u}_{k_i+p}},~ y_{s;t,a,b}^{\widetilde{u}_{k_i}, \widetilde{\alpha}_{k_i}, \widetilde{\beta}_{k_i}} \leq \sum_{p \geq 1} \theta_{k_i p} y_{s;t,a,b}^{\widehat{u}_{k_i+p}, \widehat{\alpha}_{k_i+p}, \widehat{\beta}_{k_i+p}},~ s \in [t,T],
\end{align*}
where $\preceq$ denotes the componentwise inequality. Using the Lipschitz property of $f$, $\sigma$, $\chi$ and $l$ in $u$ (see (i) of Assumption \ref{Assumption_4}) and the proof of Lemma \ref{Lemma_1}, (\ref{eq_c_4_1_1_1}) implies the convergence of the following sequence strongly in the $\mathcal{L}_{\mathbb{F}}^{\infty}$-norm sense:
\begin{align*}	
(x_{t}^{t,a;\widetilde{u}_{k_i}}, y_{t;t,a,b}^{\widetilde{u}_{k_i}, \widetilde{\alpha}_{k_i}, \widetilde{\beta}_{k_i}}) \xrightarrow{i \rightarrow \infty} (x_{t}^{t,a;\widehat{u}}, y_{t;t,a,b}^{\widehat{u}, \widehat{\alpha}, \widehat{\beta}}).
\end{align*}

By continuity of $\overline{J}$, for each $\epsilon > 0$, there exists $i^{\prime \prime}$ such that $i \geq i^{\prime \prime}$, 
\begin{align*}
\overline{J}(t,a,b;\widehat{u},\widehat{\alpha},\widehat{\beta}) \leq \overline{J}(t,a,b;	\widetilde{u}_{k_i},\widetilde{\alpha}_{k_i},\widetilde{\beta}_{k_i}) + \frac{\epsilon}{2}.
\end{align*}
This, together (ii) of Assumption \ref{Assumption_4} and (\ref{eq_c_2_1_1_1}), shows that for any $i \geq \max \{i^\prime, i^{\prime \prime}\}$,
\begin{align*}	
\overline{J}(t,a,b;\widehat{u},\widehat{\alpha},\widehat{\beta}) &\leq \overline{J}(t,a,b;	\widetilde{u}_{k_i},\widetilde{\alpha}_{k_i},\widetilde{\beta}_{k_i}) + \epsilon \\
 & \leq \mathbb{E} \Bigl [ \rho_2( \sum_{p \geq 1} \theta_{k_ip} x_{T}^{t,a;\widehat{u}_{k_i+p}}, \sum_{p \geq 1} \theta_{k_i p} y_{T;t,a,b}^{\widehat{u}_{k_i+p}, \widehat{\alpha}_{k_i+p}, \widehat{\beta}_{k_i+p}})  \\
 &\qquad \qquad + \int_t^T \rho_1(s,\sum_{p \geq 1} \theta_{k_ip} x_{s}^{t,a;\widehat{u}_{k_i+p}}, \sum_{p \geq 1} \theta_{k_i p}\widehat{u}_{k_i + p,s}) \dd s \Bigr ] + \frac{\epsilon}{2} \\
& \leq \sum_{p \geq 1} \theta_{k_i p} \overline{J}(t,a,b;	\widehat{u}_{k_i + p},\widehat{\alpha}_{k_i + p},\widehat{\beta}_{k_i + p}) + \frac{\epsilon}{2}  \\
& \leq  W(t,a,b) + \epsilon.
\end{align*}
Since $\epsilon$ is arbitrary, we have the desired result. This completes the proof.
\end{proof}

\begin{remark}
As in \cite[Appendix A]{Bokanowski_SICON_2016}, we can also use the following assumption in Theorem \ref{Theorem_C_1}:
\begin{enumerate}[(i)]
\item $f(s,x,u) = A_{s} x + B_{s} u$, $\sigma(s,x,u) = C_s x + D_s u$, $\chi(s,x,e) = E_s x + F_s u + r_s(e)$ and $l(s,x,u) = H_s x + K_s u$, 
where $A$, $B$, $C$, $E$, $F$, $r$, $H$ and $K$ are deterministic and bounded coefficients with appropriate dimensions;
\item $\rho_1$ and $\rho_2$ are convex and bounded from below;
\item $U \subset \mathbb{R}^m$ is a compact and convex set.
\end{enumerate}
Unlike the case of SDEs in a Brownian setting, there are not many results on the existence of optimal controls for jump diffusions systems. Some results related to the relaxed optimal solution approach can be found in \cite{Kushner_JMAA_2000, Hanane_AS_2017}. It is interesting to study the existence of optimal controls for jump diffusion systems in the original strong sense as for the case of SDEs driven by Brownian motion in \cite{Haussmann_SICON_1990}.
\end{remark}

\bibliographystyle{siam}
\bibliography{researches_1}

\medskip
Received xxxx 20xx; revised xxxx 20xx.
\medskip

\end{document}